\documentclass{amsart}

\usepackage{amsmath,amssymb,amsthm,amsfonts,mathrsfs,textcomp,mathtools,accents}
\usepackage{graphicx,color}
\usepackage{enumerate}
\usepackage[shortlabels]{enumitem}

\setenumerate{ label={\upshape(\roman*)},
  wide=0pt,topsep=0.2cm,labelsep=0.1em,itemsep=0.15cm,leftmargin=*}

\usepackage[headings]{fullpage}
\usepackage{hyperref}
\hypersetup{
  colorlinks=false
}

\usepackage{mleftright}
\mleftright

\DeclareFontFamily{U}{mathx}{\hyphenchar\font45}
\DeclareFontShape{U}{mathx}{m}{n}{<-> mathx10}{}
\DeclareSymbolFont{mathx}{U}{mathx}{m}{n}
\DeclareMathAccent{\widebar}{0}{mathx}{"73}

\usepackage[figuresright]{rotating}
\usepackage{tikz}
\usetikzlibrary{arrows,calc,positioning,spy}

\newtheorem{theorem}{Theorem}[section]
\newtheorem{corollary}[theorem]{Corollary}
\newtheorem{lemma}[theorem]{Lemma}
\newtheorem{proposition}[theorem]{Proposition}

\theoremstyle{definition}
\newtheorem{definition}[theorem]{Definition}

\theoremstyle{remark}
\newtheorem{remark}[theorem]{Remark}
\newtheorem*{example}{Example}

\numberwithin{equation}{section}

\newcommand{\set}[1]{\left\{#1\right\}}

\newcommand{\R}{\mathbb R}

\newcommand{\N}{\mathbb N}
\newcommand{\Z}{\mathbb Z}
\newcommand{\Q}{\mathbb Q}

\newcommand{\eps}{\varepsilon}
\newcommand{\trianglecl}{\widebar{\triangle}}
\newcommand{\couple}[4]{\left(\frac{#1}{#2},\frac{#3}{#4}\right)}
\newcommand{\triple}[3]{\left(\frac{#1}{#3},\frac{#2}{#3}\right)}
\newcommand{\singx}[3]{\left({#1},\frac{#2}{#3}\right)}
\newcommand{\singy}[3]{\left(\frac{#1}{#2},{#3}\right)}
\newcommand{\mat}[2]{\mathfrak{m}\left(\frac{#1}{#2}\right)}
\newcommand{\matt}[1]{\mathfrak{m}\left(#1\right)}
\newcommand{\mattt}[3]{\mathfrak{m}\triple{#1}{#2}{#3}}
\newcommand{\matttx}[3]{\mathfrak{m}\singx{#1}{#2}{#3}}
\newcommand{\mattty}[3]{\mathfrak{m}\singy{#1}{#2}{#3}}
\newcommand{\vet}[3]{\begin{psmallmatrix}#1\\#2\\#3\end{psmallmatrix}}
\newcommand{\Vet}[3]{\begin{pmatrix}#1\\#2\\#3\end{pmatrix}}

\DeclareMathOperator*{\diam}{diam}

\DeclareMathOperator*{\rank}{rank}

\usepackage{xcolor}
\definecolor{dgreen}{RGB}{0,160,0}    
\definecolor{orange}{RGB}{251,111,66} 
\definecolor{lblue}{RGB}{8,180,238}   
\definecolor{DBlu}{rgb}{.1,.1,.7}     

\def\centerarcfill[#1](#2)(#3:#4:#5){\draw[#1] (#2) --
  ($(#2)+({#5*cos(#3)},{#5*sin(#3)})$) arc (#3:#4:#5);}

\begin{document}

\title[]{Representation and coding of rational pairs on a Triangular tree  and
  Diophantine approximation in $\R^2$}%

\author{Claudio Bonanno}%
\address{Dipartimento di Matematica, Universit\`a di Pisa, Largo Bruno
  Pontecorvo 5, 56127 Pisa, Italy} \email{claudio.bonanno@unipi.it}

\author{Alessio Del Vigna}%
\address{Dipartimento di Matematica, Universit\`a di Pisa, Largo Bruno
  Pontecorvo 5, 56127 Pisa, Italy} \email{delvigna@mail.dm.unipi.it}

\thanks{C. Bonanno is partially supported by the research
  project PRIN 2017S35EHN$\_$004 ``Regular and stochastic
  behaviour in dynamical systems'' of the Italian Ministry of
  Education and Research, and by the ``Istituto Nazionale di Alta Matematica'' and its
  division ``Gruppo Nazionale di Fisica Matematica''. This research is part of the authors' activity within the
  DinAmicI community, see \url{www.dinamici.org}.}

\begin{abstract}
    In this paper we study the properties of the \emph{Triangular
      tree}, a complete tree of rational pairs introduced in
    \cite{cas}, in analogy with the main properties of the Farey tree
    (or Stern-Brocot tree). To our knowledge the Triangular tree is
    the first generalisation of the Farey tree constructed using the
    mediant operation. In particular we introduce a two-dimensional
    representation for the pairs in the tree, a coding which describes
    how to reach a pair by motions on the tree, and its description in
    terms of $SL(3,\Z)$ matrices. The tree and the properties we study
    are then used to introduce rational approximations of non-rational
    pairs.
\end{abstract}

\maketitle


\section{Introduction} \label{sec:intro}

The theory of multidimensional continued fractions has received
increasing attention in the last years from researchers both in number
theory and in ergodic theory. The origin of multidimensional continued
fraction expansions may be traced back to a letter that Hermite sent
to Jacobi asking for a generalisation of Lagrange's Theorem for
quadratic irrationals to algebraic irrationals of higher degree. It
was for this reason that Jacobi developed what is now called the
Jacobi-Perron algorithm. Unfortunately, despite numerous attempts and
the introduction of many different algorithms, Hermite's question
remains unanswered. We refer the reader to \cite{brentjes} for a
geometric description of the theory of multidimensional continued
fractions. On the other hand, it is well-known that (regular)
continued fraction expansions are related to the theory of dynamical
systems as the expansion of a real number can be obtained by the
symbolic representation of its orbit under the action of the Gauss map
on the unit interval. The dynamical systems approach has led to new
proofs of the Gauss-Kuzmin Theorem, Khinchin's weak law and other
metric results first obtained by Khinchin and L\'evy, also thanks to
the modern results of ergodic theory (see for example
\cite{IK,munday:iet}). In recent years the methods of ergodic theory
have been applied also to maps related to multidimensional continued
fraction algorithms, we refer to \cite{schw-book} for the first
results in this research area. It is also interesting to mention that,
in the opposite direction, number theoretical properties of real
numbers have led to new results in ergodic theory, as the Three Gaps
Theorem for instance, which can be interpreted in terms of return
times for irrational rotations on the circle.

In this paper we continue the work initiated by the authors with Sara
Munday in \cite{cas}, where we have studied the properties of a tree
of rational pairs, here called the \emph{Triangular tree}, which was
introduced as a two-dimensional version of the well-known Farey tree
(or Stern-Brocot tree). The aim of this paper is to show that all the
structures of the Farey tree can be found also in the Triangular tree
and to construct approximations of real pairs using the tree. We believe
that the Triangular tree will be useful for the study of interesting
phenomena related to two-parameter systems, as this is the case for
the Farey tree and one-parameter systems. For recent investigations on higher-dimensional phenomena we refer to \cite{berthe1,berthe2,haynes}.

The paper is structured as follows. In Section \ref{sec:farey-coding}
we recall the construction of the Farey tree and its main properties
based on the continued fraction expansions of real numbers, in
particular the $\{L,R\}$ coding of real numbers which is based on the
relations of the Farey tree with the group $SL(2,\Z)$. In Section
\ref{sec:trianglemaps} we recall the main results of \cite{cas}
concerning the Triangular tree. In particular we recall that it can be
constructed in a dynamical way, by using the Triangle map defined in
\cite{garr} which acts on the triangle
$\triangle\coloneqq\{ (x,y)\in \R^2: 1\ge x\ge y>0\}$ and its slow
version $S$ introduced in \cite{cas}, and in a geometric way, by using
the notion of mediant of two pairs of fractions. The existence of
these two possible constructions is the first basic property of the
Farey tree that has been proved for the Triangular tree in
\cite{cas}. Afterwards, Section \ref{sec:triangle-seq} contains some
preparatory results on the triangle sequences, that is the symbolic
representation of the orbits of real pairs for the Triangle map. These
sequences are the analogous of the continued fraction expansions for
real numbers and the Gauss map, and are thus the two-dimensional
continued fraction expansions of real pairs generated by the Triangle
map. It is known that there are cases in which different real pairs
have the same triangle sequence, this is discussed in details in
Section \ref{sec:triangle-seq}.

Sections \ref{sec:coding} and \ref{sec:rationals} contain the main
results of the paper. First we use the slow map $S$ and its local
inverses to introduce in Definition \ref{repres-interior} a unique
two-dimensional representation for rational pairs in $\trianglecl$,
thus improving the expansions obtained by using the Triangle map. For
non-rational pairs the two-dimensional representation is that obtained
by the Triangle map. Then in Theorem \ref{prop:interiorLRI} we
introduce a coding for rational pairs in the Triangular tree in terms
of possible motions in the tree, in analogy with the coding of the
Farey tree. This coding can be described also in terms of $SL(3,\Z)$
matrices defined in \eqref{le-tre-matrici}. This coding and its
descriptions are useful also to introduce approximations of real
pairs, and of non-rational pairs in particular, in terms of the
rational pairs on the Triangular tree. The definition of the
approximations together with some examples are described in Section
\ref{sec:approx}. Finally, in Section \ref{sec:speed} we study the
speed of the approximations introduced in the sense of the
simultaneous approximations of couples of real numbers. We give
results only for two classes of non-rational pairs, those with finite
triangle sequence and those corresponding to fixed points of the
Triangle map, leaving further developments for future research.


\section{The Farey coding}
\label{sec:farey-coding}

In this section we recall the construction of the Farey tree, the
coding it induces for the rationals in $(0,1)$, and its connection
with the continued fraction expansion of real numbers.

The \emph{Farey tree} is a binary tree which contains all the rational
numbers in the interval $(0,1)$ and can be generated in a dynamical
way using the Farey map, and in an arithmetic way using the notion of
mediant between two fractions.

The \emph{Farey map} is the map $F:[0,1]\rightarrow [0,1]$ defined to
be
\[
    F(x) \coloneqq
    \begin{cases}
        \frac{x}{1-x}\, , & \text{if } 0\leq x\leq \frac 12\\[0.2cm]
        \frac{1-x}{x}\, , & \text{if } \frac 12 \leq x \leq 1\\
    \end{cases}
\]
Denoting with $F_0:[0,\frac 12]\rightarrow [0,1]$ and
$F_1:[\frac 12,1]\rightarrow [0,1]$ its two branches, $F$ admits two local inverses,
$\psi_0\coloneqq F_0^{-1}$ and $\psi_1\coloneqq F_1^{-1}$, given by
\begin{equation}\label{farey-inversi}
    \psi_0(x) = \frac{x}{1+x}\quad\text{and}\quad
    \psi_1(x)=\frac{1}{1+x}.
\end{equation}
The Farey tree is then generated using the Farey map by setting
$\mathcal{L}_0 \coloneqq \{\frac 12\}$ for the root of the tree, and
setting recursively
$\mathcal{L}_n\coloneqq F^{-n}\left(\frac12\right)$ for all $n\ge
1$. Each level of the tree is written by ascending order as shown in
Figure \ref{fig:Farey}. The connections between the fractions of
different levels are explained below. It is known that the Farey tree
contains all the rational numbers in the interval $(0,1)$, that is
$\bigcup_{n=0}^{+\infty} \mathcal{L}_n =\bigcup_{n=0}^{+\infty}
F^{-n}\left(\frac 12\right) = \Q\cap (0,1)$, and each rational number
appears in the tree exactly once.

\begin{figure}[h]
    \begin{tikzpicture}[
    level 1/.style = {sibling distance=4cm},
    level 2/.style = {sibling distance=2cm},
    level 3/.style = {sibling distance=1cm},
    level distance          = 1cm,
    edge from parent/.style = {draw},
    scale=1.2
    ]

    \foreach \n in {0,1,2,3} {
      \pgfmathsetmacro\p{\n}
      \node at (5,-\n) {$\mathcal{L}_{\pgfmathprintnumber\p}$};
    }

    \node {$\frac 12$}
    child{
      node {$\frac 13$}
      child{
        node {$\frac 14$}
        child{
          node {$\frac 15$}
        }
        child{
          node {$\frac 27$}
        }
      }
      child{
        node {$\frac 25$}
        child{
          node {$\frac 38$}
        }
        child{
          node {$\frac 37$}
        }
      }
    }
    child{
      node {$\frac 23$}
      child{
        node {$\frac 35$}
        child{
          node {$\frac 47$}
        }
        child{
          node {$\frac 58$}
        }
      }
      child{
        node {$\frac 34$}
        child{
          node {$\frac 57$}
        }
        child{
          node {$\frac 45$}
        }
      }
    };
\end{tikzpicture}
    \caption{The first four levels of the Farey tree.}\label{fig:Farey}
\end{figure}
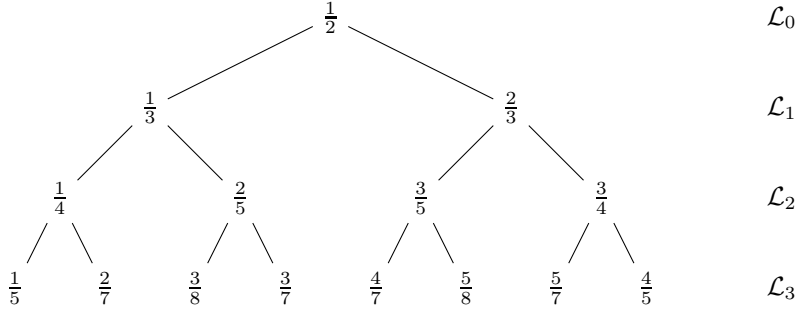

We now describe the second way to construct the levels of the Farey
tree. Let us consider the \emph{Stern-Brocot sets}
$(\mathcal{F}_n)_{n\geq -1}$, with
$\mathcal{F}_{-1}\coloneqq\left\{\frac01, \frac11\right\}$, and for
all $n\ge 0$ let $\mathcal{F}_n$ be obtained from $\mathcal{F}_{n-1}$
by inserting the mediant of each pair of neighbouring fractions. We
recall that the mediant of two fractions $\frac pq$ and $\frac rs$
with $\frac pq<\frac rs$ is
\[
    \frac pq \oplus \frac rs \coloneqq \frac{p+r}{q+s}
\]
and that $\frac pq < \frac pq \oplus \frac rs < \frac rs$. We say that
$\frac pq \oplus \frac rs$ is the \emph{child} of $\frac pq$ and
$\frac rs$, which are its \emph{left} and \emph{right parent},
respectively. It can be easily shown that two fractions
$\frac pq < \frac rs$ are neighbours in a Stern-Brocot set if and only
if $qr-ps=1$ (see, for instance, \cite{hardy}). Moreover, since the
ancestors $\frac 01$ and $\frac 11$ are in lowest terms, it follows
that all the fractions obtained through the mediant operation appear
in lowest terms. The first sets $\mathcal{F}_n$ are
\[
    \mathcal{F}_0 = \set{\frac 01, \frac 12, \frac 11},\quad
    \mathcal{F}_1 = \set{\frac 01, \frac 13, \frac 12, \frac 23, \frac
      11},\quad \mathcal{F}_2 = \set{\frac 01, \frac 14, \frac 13,
      \frac 25, \frac 12, \frac 35, \frac 23, \frac 34, \frac 11}.
\]
The levels of the Farey tree are then given by
$\mathcal{L}_n=\mathcal{F}_n\setminus\mathcal{F}_{n-1}$ for all
$n\geq 0$. Two fractions in Figure \ref{fig:Farey} are connected if
one of the fractions is a parent of the other. Notice that a fraction
in a level $\mathcal{L}_n$ has a parent in the level
$\mathcal{L}_{n-1}$ and the other parent in a level $\mathcal{L}_m$
with $m < n-1$.

For later use we recall the notion of \emph{rank} of a rational number
$x$ in the open unit interval: $\rank(x)=r$ if and only if
$x\in \mathcal{L}_{r}$. For more details on the Farey tree, we refer
to \cite{BI}.

\subsection{The Farey coding}
We now recall the $\{L,R\}$ coding of rational numbers (see also
\cite{isola:cf,BI}). Let $\frac pq\in \Q\cap (0,1)$ be a fraction
reduced in lowest terms. Since $\frac pq$ belongs to a unique level of
the Farey tree, we can write in a unique way
\[
    \frac pq = \frac lm \oplus \frac rs,
\]
with $rm-ls=1$. That is, $\frac pq$ is generated by
taking the mediant between $\frac lm$ and $\frac rs$, which are thus
its parents in the tree. We associate to the fraction $\frac pq$ the
matrix
\[
    \mat{p}{q} \coloneqq \begin{pmatrix} r&l\\s&m\end{pmatrix}\in SL(2,\Z),
\]
Introducing the two $SL(2,\Z)$ matrices
\begin{equation} \label{le-due-matrici}
    L\coloneqq \begin{pmatrix} 1&0\\1&1\end{pmatrix}
    \quad\text{and}\quad
    R\coloneqq \begin{pmatrix} 1&1\\0&1\end{pmatrix},
\end{equation}
the left and right children of $\frac pq=\frac{l+r}{m+s}$ on the tree,
that are $\frac lm \oplus \frac{l+r}{m+s}=\frac{2l+r}{2m+s}$ and
$\frac{l+r}{m+s}\oplus \frac rs = \frac{l+2r}{m+2s}$ respectively, can be
obtained by right multiplication with the matrices $L$ and $R$
respectively, since
\[
    \mat{2l+r}{2m+s}=\mat{p}{q}L = \begin{pmatrix} l+r&l\\m+s&m\end{pmatrix}
    \quad\text{and}\quad
    \mat{l+2r}{m+2s}=\mat{p}{q}R = \begin{pmatrix} r&l+r\\s&m+s\end{pmatrix}
\]
In other words, the action of $L$ and $R$ given by the
right matrix multiplication corresponds to travelling downward along the Farey
tree by moving to the left or to the right, respectively. We finally
notice that $L=\mat{1}{2}$ is the matrix of the root of the Farey
tree.

\begin{proposition}[\cite{isola:cf}]
    Let $x\in \Q\cap (0,1)$ with $\rank(x)=r$, and let
    $x=[a_1,\ldots,a_n]$ be its continued fraction expansion. Then
    \[
    \rank(x)= \sum_{i=1}^n a_i -2
    \]
    and the matrix associated
    to $x$ is
   \begin{equation}\label{m(x)cf}
        \matt{x} = \mat{1}{2} \prod_{i=1}^r M_i =  \begin{cases}
        L L^{a_1-1}R^{a_2}\cdots L^{a_{n-1}}R^{a_n-1} & \text{if $n$ is even}\\
        L L^{a_1-1}R^{a_2}\cdots R^{a_{n-1}}L^{a_n-1} & \text{if $n$ is odd}
    \end{cases}
   \end{equation}
   where $M_i=L$ if the $i$-th turn along the path joining $\frac 12$
   with $x$ on the Farey tree goes left, and $M_i=R$ if it goes
   right\footnote{The leftmost symbol $L$ denotes $\frac 12$ and is
     not associated to a move on the Farey tree.}.
\end{proposition}

The digits of the continued fraction expansion also have a dynamical
meaning in terms of the Gauss map, hence of the Farey map. In fact we
recall that the Gauss map is a \emph{fast version} of the Farey map,
precisely the Gauss map is the jump transformation of $F$ on the
interval $(\frac 12,1]$. If $x=[a_1,\ldots,a_n]$ we have
\[
    x=\psi_0^{a_1-1}\psi_1\cdots \psi_0^{a_n-1}\psi_1(0)
\]
and $(a_i)_{i=1,\,\ldots,\,n}$ is the sequence of return times to
$(\frac 12,1]$ of the orbit of $x$ under $F$. Since
$\psi_0 \psi_1(0) = \frac 12$, we can also express explicitly every
rational number $x\in (0,1)$ as a backward image of $\frac 12$ under
the Farey map. If $a_n>1$ we have
$x=\psi_0^{a_1-1}\psi_1\cdots \psi_0^{a_n-2}\left(\frac12\right)$, and
if $a_n=1$ we have
$x=\psi_0^{a_1-1}\psi_1\cdots \psi_0^{a_{n-1}-1}\left(\frac12\right)$.

\begin{example}
    Let $x=\frac{7}{12}$. This rational number appears at the fourth
    level of the Farey tree, so $x=\frac 7{12}\in \mathcal{L}_4$
    and $\rank(x)=4$. Starting from the root $\frac 12$, the path to
    reach $x$ on the Farey tree is $RLLR$, thus
    $\mat{7}{12} = LRLLR=LRL^2R$. Indeed
    \[
        LRL^2R = \begin{pmatrix} 3&4\\5&7\end{pmatrix}
    \]
    and $\frac 47\oplus \frac 35=\frac 7{12}$. Moreover we have
    $\frac 7{12} = [1,1,2,2]$, so that
    $\frac 7{12}=\psi_1\psi_1\psi_0\psi_1\left(\frac 12\right)$.
\end{example}

The coding for all the real numbers in the closed unit interval
extends the above construction and is given by a map
$\pi: [0,1]\rightarrow \{L,R\}^\N$ which associate to each
$x\in [0,1]$ an infinite sequence $\pi(x)$ over the alphabet
$\{L,R\}$. First we set
\[
    \pi\left(\frac{0}{1}\right) = L^\infty
    \quad\text{and}\quad
    \pi\left(\frac{1}{1}\right) = R^\infty.
\]
Let $x\in \Q\cap (0,1)$, so that it has a finite continued fraction
expansion, say $x=[a_1,\ldots,a_n]$. Then note that there exists
two infinite paths on the Farey tree which agree down to the node of
$x$. Both starts with the finite sequence coding the path from the
root $\frac 12$ to reach $x$, according to~\eqref{m(x)cf} and
terminating with either $RL^\infty$ or $LR^\infty$.  We let the
infinite sequence terminate with $RL^\infty$ or $LR^\infty$ according
to whether the number of partial quotients of $x$ is even or odd. Thus
for $x=[a_1,\ldots,a_n]$ we set
\[
    \pi(x)=\begin{cases}
        L^{a_1}R^{a_2}\cdots L^{a_{n-1}}R^{a_n}L^\infty & \text{if $n$ is even}\\
        L^{a_1}R^{a_2}\cdots R^{a_{n-1}}L^{a_n}R^\infty & \text{if $n$ is odd}
    \end{cases}
\]
In case $x\in (\R\setminus\Q)\cap (0,1)$ the continued fraction
expansion is infinite, say $x=[a_1,a_2,\ldots]$, thus in this case
we simply define
\[
    \pi(x)=L^{a_1}R^{a_2}L^{a_3}R^{a_4}\cdots.
\]


\section{Triangle maps and the Triangular tree} \label{sec:trianglemaps}

\subsection{The setting}
The Triangle Map has been introduced in \cite{garr} to define a
two-dimensional analogue of the continued fraction algorithm. Let us
consider the triangle
\[
    \triangle \coloneqq \set{(x,y)\in \R^2 \,:\, 1 \ge x \ge y > 0},
\]
and the pairwise disjoint subtriangles
$\triangle_k \coloneqq \set{(x,y)\in \triangle \,:\, 1-x-ky \ge 0 >
  1-x-(k+1)y}$, with $k\geq 0$, and the line segment
$\Lambda \coloneqq \set{(x,0) \,:\, 0\leq x\leq 1}$. Note that
$\trianglecl = \bigcup_{k\ge 0} \triangle_k \cup \Lambda$ (see
Figure~\ref{fig:partition}). We also introduce
\[
    \Sigma\coloneqq \trianglecl \cap \{x=y\}
    \quad\text{and}\quad
    \Upsilon\coloneqq \trianglecl\cap \{x=1\},
\]
the slanting and the vertical side of $\triangle$, respectively. The
Triangle Map $T:\triangle \to \trianglecl$ is then defined to be
\[
    \label{eq-triangle}
    T(x,y) \coloneqq \left( \frac yx, \frac{1-x-ky}{x} \right) \quad
    \text{for }(x,y)\in \triangle_k.
\]
The map $T$ generates an expansion associated to each point of
$\triangle$, the so-called \emph{triangle sequence}. In particular, to
a point $(x,y)\in \triangle$ we associate the sequence of non-negative
integers $[\alpha_0,\alpha_1,\alpha_2,\ldots]$ if and only if
$T^k(x,y)\in \triangle_{\alpha_k}$ for all $k\geq 0$. In case
$T^{k_\ast}(x,y)\in \Lambda$ for some $k_\ast>0$ then we say that the
triangle sequence terminates. An important result of \cite{garr} is
that pairs of rational numbers have a finite triangle
sequence. However, the converse is not true: also non-rational points
can have a finite triangle sequence and actually there are entire line
segments with every point having the same triangle sequence. As it is
clear from the definition, note that if $(x, y)$ has triangle sequence
$[\alpha_0,\alpha_1,\alpha_2,\ldots]$ then $T(x, y)$ has triangle
sequence $[\alpha_1,\alpha_2,\ldots]$. In other words, the Triangle
Map acts on triangle sequences as the left shift, exactly as the Gauss
map does for the continued fraction expansions.

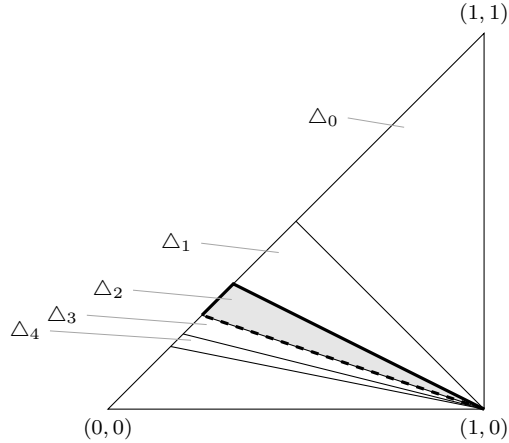
\begin{figure}[h]
    \begin{tikzpicture}[scale=2.5]
    \def\s{2}
    \def\o{0.075}

    \draw[thin] (0,0) node[below] {\footnotesize $(0,0)$} -- (\s,0)
    node[below] {\footnotesize $(1,0)$} -- (\s,\s) node[above]
    {\footnotesize $(1,1)$}-- cycle;

    \draw[very thick,fill=gray!20] (\s,0) -- (1/3*\s,1/3*\s) -- (1/4*\s,1/4*\s);
    \draw[very thick,dashed] (\s,0) -- (1/4*\s,1/4*\s);

    \foreach \x in {2,3,4,5,6}
    {
      \draw[thin] (1/\x*\s,1/\x*\s) -- (\s,0);
    }

    \foreach \x in {0,1,2,3,4}
    {
      \draw[very thin,gray!70] ({1/2*(1/(\x+1)+1/(\x+2))*\s+\o},{1/2*(1/(\x+1)+1/(\x+2))*\s}) --
      ({1/2*(1/(\x+1)+1/(\x+2))*\s-1.5*(\x+2)*\o},{1/2*(1/(\x+1)+1/(\x+2))*\s+1/2\o})
      node[black,left] {\footnotesize $\triangle_\x$};
    }
\end{tikzpicture}
    \caption{Partition of $\triangle$ into
      $\{\triangle_k\}_{k\geq 0}$.}\label{fig:partition}
\end{figure}

The analogous of the Farey map in this two-dimensional setting has
been introduced in \cite{cas}. Let $\{\Gamma_0,\Gamma_1\}$ be the
partition of $\trianglecl$ such that $\Gamma_0 \coloneqq \triangle_0$
and $\Gamma_1 \coloneqq \trianglecl\setminus \Gamma_0$. The map
$S: \trianglecl \rightarrow \trianglecl$ is then defined to be
\begin{equation}\label{eq-slow}
    S(x,y) \coloneqq
    \begin{cases}
        \left(\frac yx,\,\frac{1-x}{x}\right)\, , & (x,y)\in\Gamma_0\\[0.2cm]
        \left(\frac{x}{1-y},\,\frac{y}{1-y}\right)\, , & (x,y)\in\Gamma_1
    \end{cases}.
\end{equation}
Notice that $S$ maps $\triangle_{k}$ onto $\triangle_{k-1}$ for all
$k\ge 1$ and that $S(\Gamma_0) \cup \Sigma =
S(\Gamma_1)=\trianglecl$. As a consequence, if $(x,y)\in \triangle_k$
for some $k\geq 0$ then
\begin{equation}\label{eq:T-jumpS}
    T(x,y) = S^{k+1}(x,y) = S|_{\Gamma_0}\circ S|_{\Gamma_1}^k (x,y).
\end{equation}
In other words, the Triangle Map $T$ is the jump transformation of $S$
on the set $\Gamma_0$, as the Gauss map is the jump transformation of
the Farey map on $(\frac 12,1]$. Thus the map $S$ can be
thought of as a ``slow version'' of the Triangle Map $T$.

The map $S$ also induces a coding for the points of the triangle
$\triangle$.  In particular, if the triangle sequence of $(x,y)$ is
$[\alpha_0,\alpha_1,\ldots]$ then the itinerary under $S$ of a
point $(x,y)\in \triangle$ with respect to the partition
$\{\Gamma_0,\Gamma_1\}$ is
\[
    (\underbrace{1,\,\ldots,\,1}_{\alpha_0},\,0,
    \,\underbrace{1,\,\ldots,\,1}_{\alpha_1},\,0,\,\ldots).
\]

Many properties of the map $S$ have been proved in \cite{cas}: $S$ is
ergodic with respect to the Lebesgue measure, it preserves the
infinite Lebesgue-absolutely continuous measure with density
$\frac{1}{xy}$, and it is pointwise dual ergodic. Finally, the role of
the map $S$ as a two-dimensional version of the Farey map is confirmed
by the construction of a complete tree of rational pairs, the
\emph{Triangular tree}, by using the inverse branches of $S$, in the
same way as the Farey tree is generated by the Farey map, and then,
equivalently, by a generalised mediant operation. In
Section~\ref{sec:tree} we recall the main steps of this construction.

\subsection{Construction of the Triangular tree}
\label{sec:tree}

We now briefly recap the construction of the \emph{Triangular tree} and its
main properties, following \cite[Section~5]{cas}. The two inverse
branches of the map $S$ are
\begin{equation} \label{def-fi0}
    \phi_0 \coloneqq (S|_{\Gamma_0})^{-1} : \trianglecl\setminus
    \Sigma \rightarrow \Gamma_0,\quad \phi_0(x,y) =
    \left(\frac{1}{1+y},\ \frac{x}{1+y} \right)
\end{equation}
and
\begin{equation} \label{def-fi1}
    \phi_1\coloneqq (S|_{\Gamma_1})^{-1} : \trianglecl \rightarrow
    \Gamma_1,\quad \phi_1(x,y) = \left( \frac{x}{1+y},\
      \frac{y}{1+y}\right).
\end{equation}
We then introduce the map
\begin{equation}  \label{def-fi2}
    \phi_2: \Sigma \to \Lambda, \quad \phi_2(x,x)\coloneqq (x,0)
\end{equation}
and restrict $\phi_1$ to the set
$\trianglecl\setminus \Lambda=\triangle$. The maps $\phi_0$ and
$\phi_1$ so modified, and the map $\phi_2$ form all together the set of
local inverses of a map $\tilde S:\trianglecl \rightarrow \trianglecl$
which coincides with $S$ on $\triangle$, and satisfies
$\tilde S(x,0)=(x,x)$ on $\Lambda$. Thus the maps $S$ and $\tilde S$
coincide up to a zero-measure set.

The levels of the Triangular tree will be denoted by $(\mathcal{T}_n)_{n\geq
  -1}$. We also use the notation $\mathcal{B}_n \coloneqq
\mathcal{T}_n\cap \partial\triangle$ and $\mathcal{I}_n\coloneqq
\mathcal{T}_n\cap \accentset{\circ}{\triangle}$ for the \emph{boundary
  points} and the \emph{interior points} of the $n$-th level of the
tree, respectively. We start by setting
\[
    \mathcal{T}_{-1} \coloneqq \left\{(0,0),\,(1,0),\,(1,1)\right\}
    \quad\text{and}\quad
    \mathcal{T}_{0} \coloneqq \left\{ \singy{1}{2}{0},\,
      \singx{1}{1}{2},\, \triple{1}{1}{2} \right\}.
\]
We now describe precisely how the levels of the tree are generated, by
showing all the possibilities for taking counterimages depending on
the location of the point in $\trianglecl$ (see
Figure~\ref{fig:tree-maps} for reference).
\begin{enumerate}[label={\upshape(R\arabic*)},wide=0pt,labelsep=0.2cm,leftmargin=*]
  \item\label{rule1} An interior point
    $\triple{p}{r}{q}\in \mathcal{I}_n$ generates the two interior
    points $\triple{q}{p}{r+q}$ and $\triple{p}{r}{r+q}$ in
    $\mathcal{I}_{n+1}$, through the application of $\phi_0$ and
    $\phi_1$, respectively.
  \item\label{rule2} A boundary point
    $\triple{p}{p}{q}\in \mathcal{B}_n$ generates the point
    $\singy{p}{q}{0}\in \mathcal{B}_n$ through the application of
    $\phi_2$ and the boundary point
    $\triple{p}{p}{p+q}\in \mathcal{B}_{n+1}$ through the application
    of $\phi_1$.
  \item\label{rule3} A boundary point
    $\singy{p}{q}{0}\in \mathcal{B}_n$ generates the point
    $\singx{1}{p}{q}\in \mathcal{B}_n$ through the application of
    $\phi_0$.
  \item\label{rule4} A boundary point
    $\singx{1}{p}{q}\in \mathcal{B}_n$ generates the boundary point
    $\triple{q}{q}{p+q}\in \mathcal{B}_{n+1}$ and the interior point
    $\triple{q}{p}{p+q}\in\mathcal{I}_{n+1}$, through the application
    of $\phi_0$ and $\phi_1$, respectively.
\end{enumerate}
Note that, conversely to the Farey tree, taking a counterimage does
not necessarily imply a change in the level of the tree.

We now describe the geometric way to obtain the same two-dimensional
tree of rational pairs constructed above by counterimages (see Figure
\ref{fig:triangle_levels} for reference). We define the mediant of two
couples of fractions $\triple{p}{r}{q}$ and $\triple{p'}{r'}{q'}$ as
\[
    \triple{p}{r}{q} \oplus \triple{p'}{r'}{q'} \coloneqq
    \couple{p+p'}{q+q'}{r+r'}{q+q'}.
\]
Note that we require that the two fractions of each couple have the
same denominator, so that the mediant lies on the line segment joining
the two points it is computed from. We further assume that the two
fractions of each couple are reduced to their least common
denominator.

\begin{figure}[h]
    \begin{tikzpicture}[scale=10]
    \def\s{2}
    \def\o{0.075}
    \def\l{0.06}
    \draw[very thick] (0,0) -- (1,0) -- (1,1) -- cycle;

    \node[right] at (-\l,1-\o) {\textcolor{blue}{$\diamond$} : points of $\mathcal{T}_{-1}$};
    \node[right] at (-\l,1-\l-\o) {\textcolor{red}{\hspace{-0.05cm}\footnotesize$\triangle$\hspace{-0.05cm}} : points of $\mathcal{T}_{0}$};
    \node[right] at (-\l,1-2*\l-\o) {\textcolor{dgreen}{$\star$} : points of $\mathcal{T}_{1}$};
    \node[right] at (-\l,1-3*\l-\o) {$\circ$ : points of $\mathcal{T}_{2}$};
    \node[right] at (-\l,1-4*\l-\o) {\textcolor{gray}{$\bullet$} : points of $\mathcal{T}_3$};

    \node[blue] at (0,0) {\huge $\diamond$};
    \node[blue] at (1,0) {\huge $\diamond$};
    \node[blue] at (1,1) {\huge $\diamond$};
    \node[blue,below] at (0,0) {$(0,0)$};
    \node[blue,below] at (1,0) {$(1,0)$};
    \node[blue,above] at (1,1) {$(1,1)$};

    \node[red] at (1/2,0)   {$\triangle$};
    \node[red] at (1,1/2)   {$\triangle$};
    \node[red] at (1/2,1/2) {$\triangle$};

    \draw (1,0) -- (1/2,1/2);
    \node[dgreen] at (1/3,0)   {\huge $\star$};
    \node[dgreen] at (2/3,0)   {\huge $\star$};
    \node[dgreen] at (1,1/3)   {\huge $\star$};
    \node[dgreen] at (1,2/3)   {\huge $\star$};
    \node[dgreen] at (1/3,1/3) {\huge $\star$};
    \node[dgreen] at (2/3,2/3) {\huge $\star$};
    \node[dgreen] at (2/3,1/3) {\huge $\star$};

    \draw (1,1) -- (2/3,1/3);
    \draw (1,0) -- (1/3,1/3);
    \node at (1/4,0) {\huge $\circ$};
    \node at (2/5,0) {\huge $\circ$};
    \node at (3/5,0) {\huge $\circ$};
    \node at (3/4,0) {\huge $\circ$};
    \node at (1,1/4) {\huge $\circ$};
    \node at (1,2/5) {\huge $\circ$};
    \node at (1,3/5) {\huge $\circ$};
    \node at (1,3/4) {\huge $\circ$};
    \node at (1/4,1/4) {\huge $\circ$};
    \node at (2/5,2/5) {\huge $\circ$};
    \node at (3/5,3/5) {\huge $\circ$};
    \node at (3/4,3/4) {\huge $\circ$};
    \node at (3/5,2/5) {\huge $\circ$};
    \node at (3/4,1/4) {\huge $\circ$};
    \node at (3/4,2/4) {\huge $\circ$};
    \node at (2/4,1/4) {\huge $\circ$};

    \draw (1,0) -- (1/4,1/4);
    \draw (1/2,1/2) -- (1/2,1/4);
    \draw (1/2,1/2) -- (3/4,1/2);
    \draw (1,1) -- (3/4,1/4);
    \node[gray] at (1/5,0) {\huge $\bullet$};
    \node[gray] at (2/7,0) {\huge $\bullet$};
    \node[gray] at (3/8,0) {\huge $\bullet$};
    \node[gray] at (3/7,0) {\huge $\bullet$};
    \node[gray] at (4/7,0) {\huge $\bullet$};
    \node[gray] at (5/8,0) {\huge $\bullet$};
    \node[gray] at (5/7,0) {\huge $\bullet$};
    \node[gray] at (4/5,0) {\huge $\bullet$};
    \node[gray] at (1,1/5) {\huge $\bullet$};
    \node[gray] at (1,2/7) {\huge $\bullet$};
    \node[gray] at (1,3/8) {\huge $\bullet$};
    \node[gray] at (1,3/7) {\huge $\bullet$};
    \node[gray][gray] at (1,4/7) {\huge $\bullet$};
    \node[gray] at (1,5/8) {\huge $\bullet$};
    \node[gray] at (1,5/7) {\huge $\bullet$};
    \node[gray] at (1,4/5) {\huge $\bullet$};
    \node[gray] at (1/5,1/5) {\huge $\bullet$};
    \node[gray] at (2/7,2/7) {\huge $\bullet$};
    \node[gray] at (3/8,3/8) {\huge $\bullet$};
    \node[gray] at (3/7,3/7) {\huge $\bullet$};
    \node[gray] at (4/7,4/7) {\huge $\bullet$};
    \node[gray] at (5/8,5/8) {\huge $\bullet$};
    \node[gray] at (5/7,5/7) {\huge $\bullet$};
    \node[gray] at (4/5,4/5) {\huge $\bullet$};
    \node[gray] at (4/7,3/7) {\huge $\bullet$};
    \node[gray] at (5/8,3/8) {\huge $\bullet$};
    \node[gray] at (5/7,2/7) {\huge $\bullet$};
    \node[gray] at (4/5,1/5) {\huge $\bullet$};
    \node[gray] at (3/7,2/7) {\huge $\bullet$};
    \node[gray] at (3/5,1/5) {\huge $\bullet$};
    \node[gray] at (5/7,3/7) {\huge $\bullet$};
    \node[gray] at (4/5,3/5) {\huge $\bullet$};
    \node[gray] at (2/3,1/2) {\huge $\bullet$};
    \node[gray] at (1/2,1/3) {\huge $\bullet$};
    \node[gray] at (4/5,2/5) {\huge $\bullet$};
    \node[gray] at (2/5,1/5) {\huge $\bullet$};

    \node[right] at (0.91,0.1) {};
    \node[above,rotate=atan(2)] at (0.87,2*0.87-1) {};
    \node[above,rotate=-atan(1/2)] at (0.77,-1/2*0.77+1/2) {};
    \node at (1/2+0.1,1/2+0.025) {};
    \node[rotate=90] at (1/2-0.025,1/2-0.1) {};
    \node[rotate=atan(-1/3)] at (0.7,-1/4*0.7+1/4) {};
    \node[below,rotate=atan(3)] at (0.9,3*0.9-2) {};
\end{tikzpicture}
    \caption{The first four levels of the Triangular tree as points in
      $\trianglecl$ with the sides of the triangles of the
      partitions.}\label{fig:triangle_levels}
\end{figure}
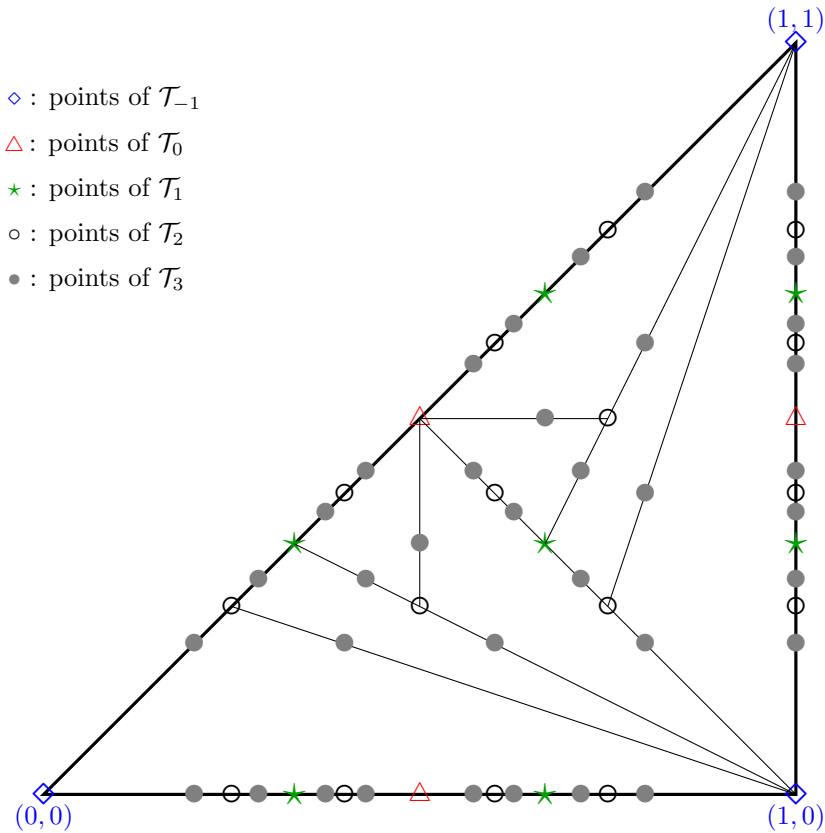

\begin{definition}
    Consider a set
    $\mathfrak{R}=\{\mathfrak{r}_i \,:\, i=1,\,\ldots,\,r\}$ of
    rational points on a line segment, consisting of at least two
    points, and in ascending lexicographic order. The \emph{Farey sum
      of $\mathfrak{R}$} is obtained by adding to $\mathfrak{R}$ the
    mediant between each pair of neighbouring points, that is
    \[
        \mathfrak{R}^\oplus \coloneqq \{\mathfrak{r}_i\oplus
        \mathfrak{r}_{i+1} \,:\, i=1,\,\ldots,\,r-1\} \cup \mathfrak{R}.
    \]
\end{definition}

\noindent To define the levels of the tree in this second, geometric
way, we start from the set
$\mathcal{S}_{-1}\coloneqq \mathcal{T}_{-1}$ of the vertices of
$\triangle$ and then we will define a sequence
$(\mathcal{S}_n)_{n\geq -1}$ of sets such that $\mathcal{S}_{-1}$ are
the three vertices of $\triangle$ and
$\mathcal{S}_n\supseteq \mathcal{S}_{n-1}$ for all $n\geq 0$. In
particular, we introduce a sequence $(\mathscr{P}_n)_{n\geq 0}$ of
measurable partitions of $\triangle$, each refining the previous one
and such that the points of $\mathcal{S}_n$ lie on the sides of the
partition $\mathscr{P}_n$. Then the recursive construction is the
following: given the set of points $\mathcal{S}_n$ up to a certain
level $n$, we obtain $\mathcal{S}_{n+1}$ by inserting the mediant
between each pair of neighbouring points along each side of the
triangles of the partition $\mathscr{P}_{n+1}$. More formally, let
$\mathscr{P}_0=\{\trianglecl\}$ and let $v_0=(0,0)$, $v_1=(1,0)$,
$v_2=(1,1)$ be the three vertices of the triangle $\triangle$. We
partition $\trianglecl$ into two subtriangles by the line segment
joining $v_1$ and $v_0\oplus v_2=\triple{1}{1}{2}$. This determines
the partition $\mathscr{P}_1$. Additionally, we label the vertices of
the two subtriangles according to the rule shown in
Figure~\ref{fig:relabel}. We now proceed inductively. Each triangle of
$\mathscr{P}_n$ is partitioned into two subtriangles by the line
segment joining the vertex labelled ``1'' with the mediant of the
vertex ``0'' and the vertex ``2'' and this gives us the next partition
$\mathscr{P}_{n+1}$. Then, for $n\geq -1$,
\[
    \mathcal{S}_{n+1} \coloneqq \bigcup_{\mathfrak{S}\in
      \mathscr{S}_{n+1}} (\mathfrak{S}\cap \mathcal{S}_n)^\oplus,
\]
where $\mathscr{S}_n$, $n\ge 0$, is the set of sides of the triangles
of the partition $\mathscr{P}_{n}$. To better understand this
construction, we recall the conclusions of \cite[Lemma~5.8]{cas} (for
simplicity of notation, we continuously extend the maps $\phi_0$ and
$\phi_1$ defined in \eqref{def-fi0} and \eqref{def-fi1} to
$\trianglecl$): for any finite binary word $\omega \in \{0,1\}^*$ of
length $n$, let $\phi_\omega$ denote the composition
$\phi_{\omega_0}\circ\phi_{\omega_1}\circ\dots\circ
\phi_{\omega_{n-1}}$, then
\begin{enumerate}
  \item the triangles of $\mathscr{P}_n$ are given by all the possible counterimages
    $\triangle_\omega\coloneqq \phi_\omega(\trianglecl)$ with
    $|\omega|=n$;
  \item
    $\mathscr{S}_{n} = \mathscr{S}_0 \cup
    \set{\widebar{\phi_\omega(\ell)} \,:\, |\omega|\leq n-1}$, where
    $\ell$ is the open line segment joining $(1,0)$ and
    $\couple{1}{2}{1}{2}$.
\end{enumerate}

\begin{figure}[h]
    \begin{tikzpicture}[scale=2.75]

    \def\t{2.5}
    \def\a{0.175}
    \def\b{0.07}
    \def\c{0.1}
    \def\f{0.8}
    \draw[thick] (0,0) -- (2/3,1) -- (7/4,-1/4) -- cycle;
    \node[blue] at ($(7/4,-1/4)+(-2.5*1.3*\b,1.5*1.2*\b)$) {\footnotesize $0$};
    \node[blue] at ($(2/3,1)+(0,-\a)$) {\footnotesize $1$};
    \node[blue] at ($(0,0)+(4/5*\a,\b)$) {\footnotesize $2$};

    \draw[thick] ($(\t,0)+(0,0)$) -- ($(\t,0)+(2/3,1)$) -- ($(\t,0)+(7/4,-1/4)$) -- cycle;
    \draw[thin] ($(\t,0)+(2/3,1)$) -- ($(\t,0)+(7/9,-1/9)$);

    \node[blue] at ($(\t,0)+(7/4,-1/4)+(-2.5*1.3*\b,1.5*1.2*\b)$) {\footnotesize $0$};
    \node[blue] at ($(\t,0)+(2/3,1)+(0.5*1.15*\a,-1.2*1.15*\a)$) {\footnotesize $1$};
    \node[blue] at ($(\t,0)+(5/6,-1/6)+(6/25*\a,2*\b)$) {\footnotesize $2$};
    \node[blue] at ($(\t,0)+(2/3,1)+(-1/1.4*1.1*\b,-1.2*1.1*\a)$) {\footnotesize $0$};
    \node[blue] at ($(\t,0)+(0,0)+(4/5*\a,\b)$) {\footnotesize $1$};
    \node[blue] at ($(\t,0)+(5/6,-1/6)+(-4.5/5*\a,2.4*\b)$) {\footnotesize $2$};

\end{tikzpicture}
    \caption{Partition of a triangle of $\mathscr{P}_n$ into two
      subtriangles and relabelling of the
      vertices.}\label{fig:relabel}
\end{figure}
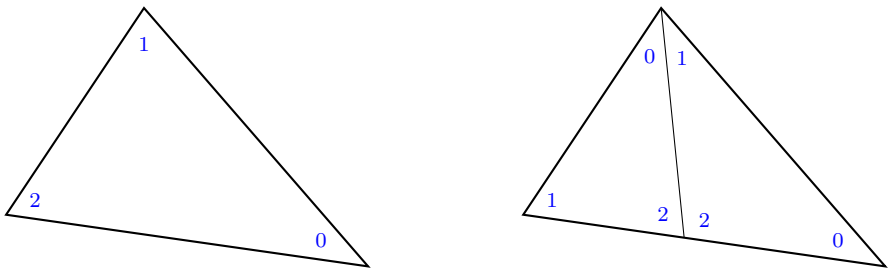

The main properties concerning the triangular tree are contained in
Theorem~5.4 and Theorem~5.8 of \cite{cas}. The first result states
that the tree is complete, that is
\[
    \bigcup_{n\geq -1} \mathcal{T}_n = \Q^2\cap \trianglecl,
\]
with every rational pair appearing exactly once in the tree. The
second result establishes the level-by-level equivalence between the
counterimages tree and the geometric tree defined above, that is
$\mathcal{T}_n = \mathcal{S}_n\setminus \mathcal{S}_{n-1}$ for all
$n\geq 0$.


\section{Triangle sequences: convergence and
  non-convergence} \label{sec:triangle-seq}

It is known that a triangle sequence does not necessarily represent a
unique pair of real numbers, but could correspond to an entire line
segment. If the triangle sequence terminates, we do not have
uniqueness and Lemma~\ref{lemma:key} characterises the points having a
given finite triangle sequence. Uniqueness is not guaranteed even when
the triangle sequence is infinite: \cite{garr} gives a sufficient
condition to have uniqueness and a criterion, equivalent to uniqueness
is proved in the later work \cite{dual-approach}. In this section we
discuss this problem.

We start by introducing some notation. Let $X=\vet{q}{p}{r}$ be a
three-dimensional vector with integer components and $q\neq 0$. We
then define the correspondent rational pair
\[
    \hat X \coloneqq \triple{p}{r}{q}
\]
for which both components have the same denominator. For instance, the
vertices $v_0$, $v_1$ and $v_2$ of $\triangle$ are represented by
$\vet{1}{0}{0}$, $\vet{1}{1}{0}$, and $\vet{1}{1}{1}$,
respectively. Note that the sum of two three-dimensional vectors
corresponds to the mediant between the two correspondent
two-dimensional vectors, that is
\[
    \widehat{X+Y} = \hat X \oplus \hat Y.
\]
For a sequence $(\alpha_0,\alpha_1,\ldots)$ of non-negative
integers and for an integer $k\geq 0$ we define
\[
    \triangle(\alpha_0,\ldots,\alpha_k) \coloneqq \set{(x,y)\in
      \triangle \,:\, T^j(x,y)\in \triangle_{\alpha_j}\text{ for all }
      j=0,\,\ldots,\,k}.
\]
The set $\triangle(\alpha_0,\ldots,\alpha_k)$ is a triangle and
consists of all those points whose first $k+1$ triangle sequence
digits are precisely $\alpha_0,\ldots,\alpha_k$ and thus these
triangles are nested, that is
\[
    \triangle \supset \triangle(\alpha_0) \supset
    \triangle(\alpha_0,\alpha_1)\supset \cdots.
\]
Let $(X_k)_{k\geq -3}$ be the
sequence of three-dimensional vectors defined as follows:
\begin{equation}\label{eq:X_k}
    X_{-3}=\Vet{0}{0}{1},\ X_{-2}=\Vet{1}{0}{0},\
    X_{-1}=\Vet{1}{1}{0},\quad X_k =
    X_{k-3}+X_{k-1}+\alpha_kX_{k-2}\text{ for all $k\geq 0$},
\end{equation}
then the vertices of $\triangle(\alpha_0,\ldots,\alpha_k)$ are
$\hat X_{k-1}$, $\hat X_k$ and $\hat X_{k-2}\oplus \hat X_k$
(see \cite[Theorem~3]{dual-approach}). Figure~\ref{fig:nested-tr}
shows the recursive construction of the triangles
$\triangle(\alpha_0,\ldots,\alpha_k)$.

When the triangle sequence
is infinite, the infinite intersection
$\bigcap_{k\geq 0} \triangle(\alpha_0,\ldots,\alpha_k)$ can be
either a point or a line segment.
\begin{enumerate}[label={\upshape(\arabic*)},wide=0pt,labelsep=0.2cm,leftmargin=*]
  \item In the first case the nested triangles shrink to a point,
    which means that the triangle sequence $(\alpha_k)_{k\geq 0}$
    denotes a unique pair of real numbers $(\alpha,\beta)$. As a
    consequence
    $\diam \triangle(\alpha_0,\ldots,\alpha_k)\rightarrow 0$ and the
    sequence $(\hat X_k)_{k\geq -3}$ converges to $(\alpha,\beta)$. We
    will refer to this case as the \emph{convergent case}.
  \item\label{point2} In the second case the triangle sequence does
    not uniquely describe a point but instead identifies a line
    segment $\mathfrak{L}$ of length $l>0$, such that all the points
    of $\mathfrak{L}$ have the same triangle sequence
    $[\alpha_0,\alpha_1,\ldots]$. In this case
    $\diam \triangle(\alpha_0,\ldots,\alpha_k)\rightarrow l$ and the
    sequence $(\hat X_k)_{k\geq -3}$ does not admit a limit. More
    precisely, we have that the odd and even terms of
    $(\hat X_k)_{k\geq -3}$ converge to the two endpoints of
    $\mathfrak{L}$ \cite[Theorem~6]{dual-approach}. In particular,
    $d(\hat X_{k-1},\hat X_k)\rightarrow l$ and it also holds that
    $d(\hat X_k,\hat X_{k-1}\oplus\hat X_k)\rightarrow 0$, where
    $d(\cdot,\cdot)$ is the Euclidean distance in $\R^2$. We will
    refer to this case as the \emph{non-convergent case}.
\end{enumerate}
The main result of \cite[Section~6]{dual-approach} is a criterion of
uniqueness, which we now state. Let
\[
    \lambda_k\coloneqq \frac{d(\hat X_{k-1},\hat X_{k+1})}{d(\hat
      X_{k-1},\hat X_k\oplus \hat X_{k-2})},
\]
and refer again to Figure~\ref{fig:nested-tr} for the geometric
interpretation of this quantity. The triangle sequence
$[\alpha_0,\alpha_1,\ldots]$ does not correspond to a unique pair
of real numbers if and only if it contains only a finite number of zeroes
and $\prod_{k\geq N}(1-\lambda_k)>0$, where $N$ is such that
$\alpha_k>0$ for all $k\geq N$\footnote{Note that $\lambda_k=1$ if and
  only if $\alpha_k=0$.}. The convergence of the infinite product to a
non-zero number is equivalent to the convergence of
$\sum_{k\geq 0} \lambda_k$, which is in turn equivalent to
$\lambda_k\rightarrow 0$ sufficiently fast: the geometric meaning of
$\lambda_k$ suggests that this condition is equivalent to a
sufficiently fast growth of the triangle sequence digits $\alpha_k$.

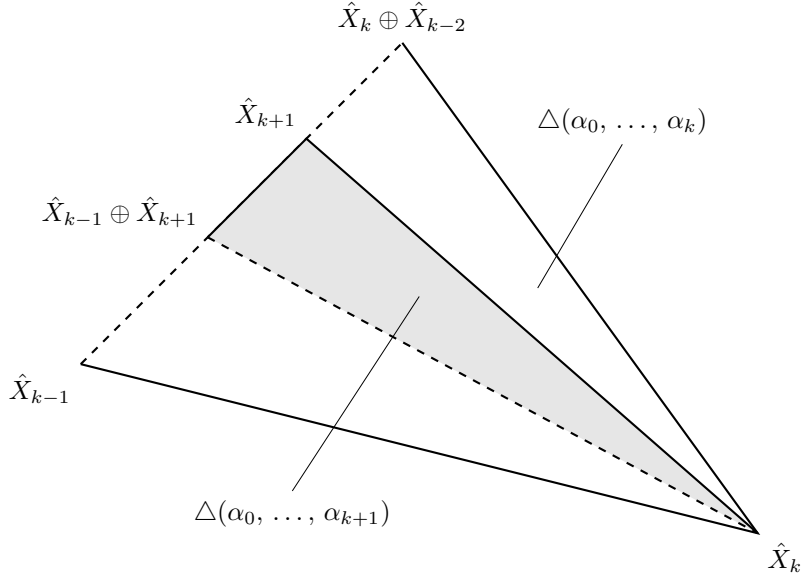
\begin{figure}[h]
    \begin{tikzpicture}[scale=2.25]

    \draw[thick] (0,0) node[below left]{$\hat X_{k-1}$}  -- (4,-1) node[below right]{$\hat X_k$} -- (1.9,1.9) node[above]{$\hat X_k\oplus \hat X_{k-2}$};

    \draw[thick,dashed] (0,0) -- (1.9,1.9);

    \draw[thick,fill=gray!20] (4,-1) -- (4/3,4/3) node[above left] {$\hat X_{k+1}$} -- (3/4,3/4) node[above left] {$\hat X_{k-1} \oplus \hat X_{k+1}$};

    \draw[thick,dashed] (3/4,3/4) -- (4,-1);

    \draw[very thin] (3.2,1.3)--(2.7,0.45);
    \node[above] at (3.2,1.3) {$\triangle(\alpha_0,\,\ldots,\,\alpha_k)$};

	\draw[very thin] (1.25,-0.75)--(2,0.4);
	\node[below] at (1.25,-0.75) {$\triangle(\alpha_0,\,\ldots,\,\alpha_{k+1})$};

\end{tikzpicture}
    \caption{Construction of the triangle
      $\triangle(\alpha_0,\,\ldots,\,\alpha_k,\,\alpha_{k+1})$
      starting from
      $\triangle(\alpha_0,\,\ldots,\,\alpha_k)$·}\label{fig:nested-tr}
\end{figure}

We now prove some results of convergence for the points of the closed
triangle $\trianglecl$ in the non-convergent case. For two rational
pairs $\hat X$ and $\hat Y$ and a non-negative integer $s$ we introduce
the notation
\[
    \hat Y \oplus_s \hat X \coloneqq \hat Y \oplus \underbrace{\hat X
      \oplus \cdots \oplus \hat X}_{s\text{ times}},
\]
and recall that the maps $\phi_0$ and $\phi_1$ commute with the
mediant operation for rational pairs whose components have the same
denominator, that is
\[
    \phi_i\left( \hat X \oplus \hat Y\right) = \phi_i(\hat X) \oplus
    \phi_i(\hat Y)\quad\text{for $i=0,1$}.
\]
Also, in the following we denote by $\phi_0$ and $\phi_1$ the
continuous extension to $\trianglecl$ of the maps defined in
\eqref{def-fi0} and \eqref{def-fi1}.

\begin{lemma}\label{lemma:delta_0k}
    Let $k\geq 0$ and let $(\alpha_j)_{j=0,\,\ldots,\,k}$ be a
    sequence of non-negative integers. It holds that
    \[
        \triangle(\alpha_0,\ldots,\alpha_k) =
        \phi_1^{\alpha_0}\phi_0\cdots
        \phi_1^{\alpha_k}\phi_0(\trianglecl\setminus \Sigma),
    \]
    and more precisely
    $\hat X_k = \phi_1^{\alpha_0}\phi_0\cdots
    \phi_1^{\alpha_k}\phi_0(1,0)$.
\end{lemma}
\begin{proof}
    Let $d$ be a non-negative integer. We have
    $\triangle_d = \phi_1^d\phi_0(\trianglecl\setminus \Sigma)$,
    which can be verified by explicitly computing the vertices of the
    two triangles. This in turn proves the result in the case $k=0$.
    \\[0.15cm]
    For $k\geq 1$, by definition we have
    $(x,y)\in \triangle(\alpha_0,\ldots,\alpha_k)$ if and only if
    \[
        T^j(x,y)\in \triangle_{\alpha_j}\quad\text{for all
        }j=0,\,\ldots,\,k.
    \]
    We can thus write
    \[
        T^k(x,y)=T|_{\triangle_{\alpha_{k-1}}}\circ \cdots \circ
        T|_{\triangle_{\alpha_0}}(x,y) \overset{\eqref{eq:T-jumpS}}{=}
        S|_{\Gamma_0} S|_{\Gamma_1}^{\alpha_{k-1}}\circ\cdots\circ
        S|_{\Gamma_0}\circ S|_{\Gamma_1}^{\alpha_0}(x,y)\in
        \triangle_{\alpha_k}
    \]
    which holds if and only if
    $(x,y)\in \phi_1^{\alpha_0}\phi_0\cdots
    \phi_1^{\alpha_{k-1}}\phi_0(\triangle_{\alpha_k})$. Since
    $\triangle_{\alpha_k}=\phi_1^{\alpha_k}\phi_0(\trianglecl\setminus
    \Sigma)$, the first part of the result follows.\\[0.15cm]
    For the second part of the lemma we argue by strong induction on
    $k\geq 0$. The case $k=0$ follows from the first part of this
    proof, by the explicit computation of the vertices of
    $\triangle(\alpha_0)$. Now let $k\geq 1$ and consider the triangle
    $\triangle(\alpha_0,\ldots,\alpha_{k-1})=
    \phi_1^{\alpha_0}\phi_0\cdots
    \phi_1^{\alpha_{k-1}}\phi_0(\trianglecl\setminus \Sigma)$.  By
    inductive hypothesis two of its vertices are
    $\hat X_{k-1}=\phi_1^{\alpha_0}\phi_0\cdots
    \phi_1^{\alpha_{k-1}}\phi_0(1,0)$ and
    $\hat X_{k-2} =\phi_1^{\alpha_0}\phi_0\cdots
    \phi_1^{\alpha_{k-2}}\phi_0(1,0)$. Since $\phi_0(0,0)=(1,0)$ and
    $\phi_1(1,0)=(1,0)$, we can also write
    $\hat X_{k-2} = \phi_1^{\alpha_0}\phi_0\cdots
    \phi_1^{\alpha_{k-1}}\phi_0(0,0)$. Thus, from
    $\triangle(\alpha_0,\ldots,\alpha_{k-1}) =
    \phi_1^{\alpha_0}\phi_0\cdots
    \phi_1^{\alpha_{k-1}}\phi_0(\trianglecl\setminus \Sigma)$, it
    follows that the other vertex of
    $\triangle(\alpha_0,\ldots,\alpha_{k-1})$ has to be the backward
    image of $(1,1)$, that is
    $\hat X_{k-3}\oplus \hat X_{k-1}= \phi_1^{\alpha_0}\phi_0\cdots
    \phi_1^{\alpha_{k-1}}\phi_0(1,1)$. Since the maps $\phi_0$ and
    $\phi_1$ commute with the mediant, we have
    \begin{align*}
        \hat X_k &= (\hat X_{k-3}\oplus \hat X_{k-1})\oplus_{\alpha_k}
        \hat X_{k-2}= \phi_1^{\alpha_0}\phi_0\cdots
        \phi_1^{\alpha_{k-1}}\phi_0(1,1)\oplus_{\alpha_k}
        \phi_1^{\alpha_0}\phi_0\cdots \phi_1^{\alpha_{k-1}}\phi_0(0,0)
        =\\
        &= \phi_1^{\alpha_0}\phi_0\cdots
        \phi_1^{\alpha_{k-1}}\phi_0\left(\frac 1{\alpha_k+1},\frac
          1{\alpha_k+1}\right)= \phi_1^{\alpha_0}\phi_0\cdots
        \phi_1^{\alpha_{k-1}}\phi_0\phi_1^{\alpha_k}\phi_0(1,0),
    \end{align*}
    and this completes the proof.
\end{proof}

\begin{lemma}\label{lemma:s-farey-sum}
    Let $\hat X = \triple{p}{r}{q}$ and $\hat Y = \triple{p'}{r'}{q'}$,
    and let $s$ be a non-negative integer. Then
    \[
        d(\hat Y\oplus_s \hat X,\hat X) = d(\hat X,\hat
        Y)\frac{q'}{q'+sq} \quad\text{and}\quad d(\hat Y\oplus_s \hat
        X,\hat Y) = d(\hat X,\hat Y)\frac{sq}{q'+sq}.
    \]
\end{lemma}
\begin{proof}
    It is a straightforward computation involving $\hat X$, $\hat Y$
    and $\hat Y\oplus_s\hat X = \triple{p'+sp}{r'+sr}{q'+sq}$.
\end{proof}

For $k\geq -3$ we introduce the notation
\[
    X_k = \Vet{q_k}{p_k}{r_k},
\]
so that $\hat X_k=\triple{p_k}{r_k}{q_k}$. Note that \eqref{eq:X_k}
implies that the denominators satisfy the recurrence
\[
    q_k=q_{k-3}+q_{k-1}+\alpha_kq_{k-2},
\]
where $q_{-3}=0$, $q_{-2}=1$ and $q_{-1}=1$. This remark and
Lemma~\ref{lemma:s-farey-sum} imply that for all $k\geq 0$ we have
\[
    \lambda_k = \frac{d(\hat X_{k-1},\hat X_{k+1})}{d(\hat
      X_{k-1},\hat X_k\oplus \hat X_{k-2})} =
    \frac{q_{k-2}+q_k}{q_{k-2}+q_k+\alpha_{k+1}q_{k-1}} =
    \frac{q_{k-2}+q_k}{q_{k+1}}.
\]
It also easily follows that
\begin{equation}\label{eq:1-lambda}
    1-\lambda_k =\alpha_{k+1}\frac{q_{k-1}}{q_{k+1}}.
\end{equation}

\begin{lemma}\label{lemma:limit-dist}
    Let $(\alpha_k)_{k\geq 0}$ be a non-convergent triangle sequence
    describing a line segment of length $l>0$ (as described in point
    \ref{point2} above).
    \begin{enumerate}
      \item The ratio of consecutive denominators of the $X_k$ diverges, that is
        $\lim_{k\rightarrow +\infty} \frac{q_k}{q_{k-1}} =+\infty$.
      \item It holds
        \[
            \lim_{k\rightarrow+\infty}d(\hat X_{k-2}\oplus \hat
            X_k,\hat X_{k-1})=l.
        \]
      \item For any non-negative integer $s$, it holds
        \[
            \lim_{k\rightarrow +\infty} d(\hat X_k \oplus_s \hat
            X_{k-1},\hat X_k) = 0 \quad\text{and}\quad
            \lim_{k\rightarrow +\infty} d((\hat X_{k-2}\oplus \hat
            X_k)\oplus_s \hat X_{k-1},\hat X_{k-2}\oplus \hat X_k)=0.
        \]
    \end{enumerate}
\end{lemma}
\begin{proof}
    (i) From the recurrence for the denominators we have
    \[
        \frac{q_k}{q_{k-1}} \geq 1+\alpha_k \frac{q_{k-2}}{q_{k-1}}
        \overset{\textrm{\eqref{eq:1-lambda}}}{=}
        1+(1-\lambda_{k-1})\frac{q_k}{q_{k-1}},
    \]
    so that $\frac{q_k}{q_{k-1}} \geq \frac{1}{\lambda_k}$. Since
    $\lambda_k\rightarrow 0^+$ as $k\rightarrow +\infty$, the thesis
    follows.\\[0.15cm]
    (ii) For $k$ large enough we have
    \[
        d(\hat X_{k-1},\hat X_k)\leq d(\hat X_{k-2}\oplus \hat
        X_k,\hat X_{k-1})\leq d(\hat X_{k-2}\oplus \hat X_k,\hat
        X_{k})+ d(\hat X_{k-1},\hat
        X_k),
    \]
    where the first inequality is shown in
    \cite[Theorem~6]{dual-approach} and the second inequality is the
    triangle inequality applied to
    $\triangle(\alpha_0,\,\ldots,\,\alpha_k)$. This is enough to
    conclude because we already know that
    $d(\hat X_{k-1},\hat X_k)\rightarrow l$ and
    $d(\hat X_{k-2}\oplus \hat X_k,\hat X_{k})\rightarrow 0$ (see point \ref{point2} above).\\[0.15cm]
    (iii) From Lemma~\ref{lemma:s-farey-sum} we have
    \[
        d(\hat X_k \oplus_s \hat X_{k-1},\hat X_k) = d(\hat X_k,\hat
        X_{k-1})\frac{s}{s+\frac{q_k}{q_{k-1}}}
    \]
    and
    \[
        d((\hat X_{k-2}\oplus \hat X_k)\oplus_s \hat X_{k-1},\hat
        X_{k-2}\oplus \hat X_k) =d(\hat X_{k-2}\oplus \hat X_k,\hat
        X_{k-1})\frac{s}{s+\frac{q_k}{q_{k-1}} +
          \frac{q_{k-2}}{q_{k-1}}}.
    \]
    Using $d(\hat X_{k-1},\hat X_k)\rightarrow l$, (i), and (ii), the
    result easily follows.
\end{proof}

\begin{proposition}\label{prop:nonconvergent}
    Let $(\alpha_k)_{k\geq 0}$ be a non-convergent triangle sequence
    and denote by $\mathfrak{p}_{\mathfrak L}$ and
    $\mathfrak{q}_{\mathfrak{L}}$ the two endpoints of the line
    segment
    $\mathfrak{L} = \bigcap_{k\geq 0}
    \triangle(\alpha_0,\,\ldots,\,\alpha_k)$, such that
    $(\hat X_{2k})_{k\geq 0}$ and $(\hat X_{2k+1})_{k\geq 0}$ converge
    respectively to $\mathfrak{p}_{\mathfrak L}$ and
    $\mathfrak{q}_{\mathfrak{L}}$. Then for all
  $(x,y)\in \trianglecl\setminus\set{(0,0)}$ it holds
    \[
        \lim_{k\rightarrow+\infty} \phi_1^{\alpha_0}\phi_0\cdots
        \phi_1^{\alpha_{2k}}\phi_0(x,y) = \mathfrak{p}_{\mathfrak{L}},\quad
         \lim_{k\rightarrow+\infty} \phi_1^{\alpha_0}\phi_0\cdots
        \phi_1^{\alpha_{2k+1}}\phi_0(x,y) = \mathfrak{q}_{\mathfrak{L}},
            \]
    and
    \[
       \lim_{k\rightarrow+\infty} \phi_1^{\alpha_0}\phi_0\cdots
        \phi_1^{\alpha_{2k}}\phi_0(0,0) = \mathfrak{q}_{\mathfrak{L}},\quad
        \lim_{k\rightarrow+\infty} \phi_1^{\alpha_0}\phi_0\cdots
        \phi_1^{\alpha_{2k+1}}\phi_0(0,0) = \mathfrak{p}_{\mathfrak{L}}.
    \]
\end{proposition}
\begin{proof}
    We give the proof just for the case of even indices, the odd case
    is analogous. By notation we have that
    $\hat X_{2k} = \phi_1^{\alpha_0}\phi_0\cdots
    \phi_1^{\alpha_{2k}}\phi_0(1,0)$ converges to
    $\mathfrak{p}_{\mathfrak L}$ and that
    $\hat X_{2k+1} = \phi_1^{\alpha_0}\phi_0\cdots
    \phi_1^{\alpha_{2k+1}}\phi_0(1,0)$ converges to
    $\mathfrak{q}_{\mathfrak L}$. Using that
    $(1,0)=\phi_1 \phi_0(0,0)$, we can write
    $\hat X_{2k+1} = \phi_1^{\alpha_0}\phi_0\cdots
    \phi_1^{\alpha_{2k+2}}\phi_0(0,0)$, so that
    $\lim_{k\rightarrow+\infty} \phi_1^{\alpha_0}\phi_0\cdots
    \phi_1^{\alpha_{2k}}\phi_0(0,0) = \mathfrak{q}_{\mathfrak{L}}$.\\[0.15cm]
    Let now $(x,y)\in \trianglecl\setminus\set{(0,0)}$ and notice that
    if
    $\lim_{k\rightarrow+\infty} \phi_1^{\alpha_0}\phi_0\cdots
    \phi_1^{\alpha_{2k}}\phi_0(x,y)$ exists then it must lie on the
    line segment $\mathfrak{L}$. We start considering the case when
    $(x,y)$ is on the boundary of
    $\trianglecl\setminus\{(0,0)\}$. Suppose that
    $(x,y)\in \Lambda\setminus\set{(0,0)}$, so that $(x,y)=(\xi,0)$
    for a certain $\xi\in (0,1]$, and let $s$ be a positive integer
    such that $\frac 1s<\xi$. Thus
    $(\frac 1s, 0) = (1,0)\oplus_{s-1} (0,0)$ . Writing $\hat X_{2k}$
    and $\hat X_{2k-1}$ as above, we have
    \[
        \phi_1^{\alpha_0}\phi_0\cdots
        \phi_1^{\alpha_{2k}}\phi_0\left(\frac 1s,0\right) =
        \phi_1^{\alpha_0}\phi_0\cdots
        \phi_1^{\alpha_{2k}}\phi_0(1,0)\oplus_{s-1}
        \phi_1^{\alpha_0}\phi_0\cdots \phi_1^{\alpha_{2k}}\phi_0(0,0)
        = \hat X_{2k}\oplus_{s-1} \hat X_{2k-1}.
    \]
    where we have used that $\phi_0$ and $\phi_1$ commute with the
    mediant operation. Moreover $\phi_0$ and $\phi_1$ are monotonic
    along line segments (with respect to the lexicographic order) and
    $(\frac 1s,0)\leq_{lex} (\xi,0)\leq_{lex} (1,0)$, hence
    \[
        0\leq d(\phi_1^{\alpha_0}\phi_0\cdots
        \phi_1^{\alpha_{2k}}\phi_0\left(\xi,0\right),\hat X_{2k})\leq
        d\left(\phi_1^{\alpha_0}\phi_0\cdots
          \phi_1^{\alpha_{2k}}\phi_0\left(\frac 1s,0\right),\hat
          X_{2k}\right)= d\left(\hat X_{2k}\oplus_{s-1} \hat X_{2k-1},\hat
          X_{2k}\right).
    \]
    Lemma~\ref{lemma:limit-dist}-(iii) gives
    $d(\hat X_{2k}\oplus_{s-1} \hat X_{2k-1},\hat
      X_{2k})\rightarrow 0$ for $k\rightarrow +\infty$, thus
    $d(\phi_1^{\alpha_0}\phi_0\cdots
      \phi_1^{\alpha_{2k}}\phi_0(\xi,0),\hat X_{2k})\rightarrow
    0$, which means that
    $\lim_{k\rightarrow+\infty} \phi_1^{\alpha_0}\phi_0\cdots
    \phi_1^{\alpha_{2k}}\phi_0(\xi,0)=\mathfrak{p}_{\mathfrak L}$. If
    $(x,y)\in \Upsilon$ or $(x,y)\in \Sigma\setminus\{(0,0)\}$ we can
    argue as above to conclude that
    $\lim_{k\rightarrow+\infty} \phi_1^{\alpha_0}\phi_0\cdots
    \phi_1^{\alpha_{2k}}\phi_0(x,y)=\mathfrak{p}_{\mathfrak L}$.\\[0.15cm]
    Finally, let $(x,y)$ be an interior point of the triangle
    $\triangle$ and let $A$ be a triangle containing $(x,y)$ as
    interior point and having all the vertices along
    $\partial \triangle\setminus\{(0,0)\}$. The maps $\phi_0$ and
    $\phi_1$ map triangles into triangles and the same does the
    composite map
    $\phi_1^{\alpha_0}\phi_0\cdots \phi_1^{\alpha_{2k}}\phi_0$. Thus
    for all $k\geq 0$ the image
    $\phi_1^{\alpha_0}\phi_0\cdots \phi_1^{\alpha_{2k}}\phi_0(A)$ is a
    triangle containing
    $\phi_1^{\alpha_0}\phi_0\cdots \phi_1^{\alpha_{2k}}\phi_0(x,y)$ as
    an interior point. The thesis follows by observing that
    $\phi_1^{\alpha_0}\phi_0\cdots \phi_1^{\alpha_{2k}}\phi_0(A)$
    shrinks to $\mathfrak{p}_{\mathfrak{L}}$ because its three
    vertices converge to $\mathfrak{p}_{\mathfrak{L}}$.
\end{proof}

The last result will be important in Section~\ref{sec:coding} to
better understand the coding of real pairs in the non-convergent case.


\section{A two-dimensional representation}
\label{sec:coding}

In this section we begin to use our construction of the Triangular
tree and the properties of the maps $\phi_0$, $\phi_1$ and $\phi_2$
defined in \eqref{def-fi0}-\eqref{def-fi2} to introduce a new
representation of real pairs of numbers in $\trianglecl$ by combining
triangle sequences and continued fraction expansions. We recall that
for any finite binary word $\omega \in \{0,1\}^*$ of length $n$, we
let
$\phi_\omega
\coloneqq\phi_{\omega_0}\circ\phi_{\omega_1}\circ\dots\circ
\phi_{\omega_{n-1}}$.

\begin{lemma}\label{lemma:key}
    Let $(\alpha,\beta)$ be a point of $\trianglecl$ with finite
    triangle sequence $[\alpha_0\,\,\ldots,\,\alpha_k]$. If
    $(\alpha,\beta)$ is an interior point, then:
     \begin{enumerate}
       \item
         $(\alpha,\beta)\in
         \phi_{\omega}\phi_2(\Sigma\setminus\{(0,0),(1,1)\})$, where
         $\omega = 1^{\alpha_0}0\cdots
         1^{\alpha_{k-1}}01^{\alpha_k}0$, so that there exists a
         unique $\xi\in (0,1)$ such that
         \begin{equation}\label{eq:backSigma}
             (\alpha,\beta) =
             \phi_1^{\alpha_0}\phi_0\cdots\phi_1^{\alpha_k}\phi_0\phi_2(\xi,\xi);
         \end{equation}
       \item $\xi$ is rational if and only if
         $(\alpha,\beta)\in \Q^2$.
     \end{enumerate}
\end{lemma}
\begin{proof}
    (i) By definition of triangle sequences we have
    \[
        T^j(\alpha,\beta)\in \triangle_{\alpha_j}\quad\text{for all
          $j=0,\,\ldots,\,k$}\quad\text{and}\quad
        T^{k+1}(\alpha,\beta)\in\Lambda.
    \]
    Thus $T^{k+1}(\alpha,\beta)=(\xi,0)=\phi_2(\xi,\xi)$ for some
    $\xi\in (0,1)$. From~\eqref{eq:T-jumpS} and arguing as in
    Lemma~\ref{lemma:delta_0k} the thesis follows.\\[0.1cm]
    (ii) From~(i) we have that $(\alpha,\beta)$ can be obtained from
    $(\xi,0)$ with the application of a finite number of maps between
    $\phi_0$ and $\phi_1$, which are linear fractional maps.
\end{proof}

\begin{remark}
    Let $(\alpha,\beta)$ be an interior point having finite triangle
    sequence $[\alpha_0,\,\ldots,\,\alpha_k]$. Note that $\alpha_k>0$
    because the only points in $\trianglecl$ for which the triangle
    sequence is defined and ends with $0$ are located on
    $(\Sigma\cup \Upsilon) \setminus\Lambda$. In light of this simple
    remark we can rewrite~\eqref{eq:backSigma} as
    \[
        (\alpha,\beta)=\phi_1^{\alpha_0}\phi_0 \cdots
        \phi_1^{\alpha_k-1}\left(\frac 1{1+\xi},\frac{\xi}{1+\xi}\right),
    \]
    so that $(\alpha,\beta)\in \phi_\omega(\ell)$ where
    $\omega = 1^{\alpha_0}0\cdots 1^{\alpha_k-1}$.
\end{remark}

\begin{definition} \label{repres-interior}
    Let $(\alpha,\beta)$ be an interior point of $\trianglecl$ with finite
    triangle sequence $[\alpha_0,\,\ldots,\,\alpha_k]$, and let
    $[a_1,\,a_2,\,\ldots]$ be the continued fraction expansion of the
    unique number $\xi$ given by Lemma~\ref{lemma:key}-(ii). We
    associate to $(\alpha,\beta)$ the \emph{representation} given by the
    pair
    \[
        \Big([\alpha_0,\ldots,\alpha_k],[a_1,\ldots]\Big).
    \]
    As for the second component, we write $[0]$ and $[1]$ to
    denote the continued fraction expansion of $0$ and $1$,
    respectively.
\end{definition}

\begin{remark}\label{remark:ell}
    If $(\alpha,\beta)$ is a rational pair, then $\xi$ is rational
    (Lemma~\ref{lemma:key}-(iii)) and thus its continued fraction
    expansion is finite, say $[a_1,\,\ldots,\,a_n]$. In this case we
    further assume that $a_n>1$ when $n>1$.
\end{remark}

\noindent We give further details for the coding of boundary rational
points and, in particular, for the vertices of the triangle
$\triangle$.
\begin{enumerate}[label={\upshape(\arabic*)},wide =
    0pt,leftmargin=*,labelsep=0.2cm]
  \item A point on $\Sigma\setminus\{(0,0),(1,1)\}$ is of the kind
    $(\xi,\xi)$, for some $\xi$ in the open unit interval. If
    $[a_1,\,\ldots]$ is the continued fraction expansion of $\xi$,
    then $\frac{1}{a_1-1}<\xi\leq \frac 1{a_1}$, so that
    $(\xi,\xi)\in \triangle_{a_1-1}$. In case $\xi=\frac{1}{a_1}$ we
    have $T(\xi,\xi)=(1,0)\in \Lambda$, so that the triangle sequence
    of $(\xi,\xi)$ is $[a_1-1]$. Otherwise, if
    $\frac 1{a_1-1}<\xi<\frac 1{a_1}$, we have
    $T(\xi,\xi)\in \Upsilon\setminus\{(1,0)\}$, so that the triangle
    sequence of $(\xi,\xi)$ is $[a_1-1,0]$. Thus we set the representation of
    $(\xi,\xi)$ to be respectively
    \[
        \Big([a_1-1],[a_1]\Big) \quad\text{or}\quad
        \Big([a_1-1,0],[a_1,\,\ldots]\Big).
    \]
  \item A point in $\Lambda$ is of the kind $(\xi,0)$, for some $\xi$
    in the unit interval, and its triangle sequence is not defined and
    assumed to be empty. We thus represent $(\xi,0)$ with
    \[
        \Big([],[a_1,\,\ldots]\Big),
    \]
    where $[a_1,\,\ldots]$ is the continued fraction expansion
    of $\xi$. In particular, the representation of the vertices $(0,0)$ and
    $(1,0)$ are $\left([],[0]\right)$ and $\left([],[1]\right)$,
    respectively.
  \item A point in $\Upsilon\setminus\{(1,0)\}$ is of the kind
    $(1,\xi)$, for some $\xi\in (0,1]$, and has triangle sequence
    $[0]$. The representation of $(1,\xi)$ is thus
    \[
        \Big([0],[a_1,\,\ldots]\Big),
    \]
    where $[a_1,\,\ldots]$ is the continued fraction expansion
    of $\xi$. In particular, the representation of the vertex $(1,1)$ is
    $\left([0],[1]\right)$.
\end{enumerate}

\vskip 0.2cm The representation of real pairs with infinite triangle
sequences depends on the convergence of the sequence. We have seen in
Section \ref{sec:triangle-seq} that if $(\alpha,\beta)$ is a real pair
in $\trianglecl$ with convergent infinite triangle sequence
$[\alpha_0,\alpha_1,\dots]$, then
\[
    \{(\alpha,\beta)\} = \bigcap_{k\ge 0}\, \triangle(\alpha_0,\,\ldots,\,\alpha_k)
\]
This shows that in this case it is enough to associate to $(\alpha,\beta)$ its triangle sequence, since
\begin{equation} \label{converg-ok}
\lim_{k\to \infty}\, \phi_1^{\alpha_0}\phi_0\cdots\phi_1^{\alpha_k}\phi_0 (x,y) =  (\alpha,\beta)
\end{equation}
for all $(x,y) \in \trianglecl$. In the definition below we choose
$(x,y)=(\frac 12,0)$, but other choices would work as well.

\begin{definition} \label{repres-infinite-conv} Let $(\alpha,\beta)$
    be a point of $\trianglecl$ with convergent infinite triangle
    sequence $[\alpha_0,\alpha_1,\ldots]$. We associate to
    $(\alpha,\beta)$ the \emph{representation} given by the pair
    \[
        \Big([\alpha_0,\alpha_1,\ldots],[2]\Big).
    \]
\end{definition}

Let us now consider the case of non-convergent infinite triangle
sequences. In this case a line segment $\mathfrak{L}$ is associated to
such a sequence, and we refer to Proposition \ref{prop:nonconvergent}
for the notation of its endpoints and more properties. Using these
results we give the following
definition. 

\begin{definition}
    \label{repres-infinite-nonconv} Let $(\alpha,\beta)$ be a point of
    $\trianglecl$ with non-convergent infinite triangle sequence
    $[\alpha_0,\alpha_1,\ldots]$. Then $(\alpha,\beta)$ belongs to a
    line segment $\mathfrak{L}$ of real pairs having the same triangle
    sequence, with endpoints $\mathfrak{p}_\mathfrak{L}$ and
    $\mathfrak{q}_\mathfrak{L}$. Then we consider
    \[
    	\Big([\alpha_0,\alpha_1,\ldots],[2]\Big).
    \]
    to be the \emph{representation} of the segment $\mathfrak{L}$.
\end{definition}


\section{The Triangular coding}
\label{sec:rationals}

In this section we use the ideas exposed in Section \ref{sec:coding}
to introduce a coding for the rational pairs in the Triangular tree in
an analogous way as for the Farey coding.  As recalled in
Section~\ref{sec:farey-coding}, the continued fraction expansion of a
rational number in $(0,1)$ is related to the path on the Farey tree to
reach it starting from the root $\frac 12$. This information is
contained in the $\{L,R\}$ coding, which in turn can be seen as the
action by right multiplication of two matrices $L,R\in SL(2,\Z)$ on
the matrix representation of rational numbers. We now generalise this
setting for the two-dimensional case by first defining a coding for
the possible moves from parents to children along the Triangular tree
in figure \ref{fig:triangle_levels}. Then we introduce a matrix
representation for rational pairs and convert the action of the moves
on the tree into the action by right multiplication of specific
$SL(3,\Z)$ matrices.

\begin{definition}
    \label{def:actionsLR-2d} Let $\mathfrak{S}$ be a line segment in
    $\trianglecl$ obtained as a counterimage of
    $\Sigma\setminus\set{(0,0),(1,1)}$ by a combination
    $\phi_{\tilde \omega}$ of the maps $\phi_0$, $\phi_1$ and
    $\phi_2$, where $\tilde\omega$ is a finite word which is either
    empty or is the concatenation $(\omega 2)$ with
    $\omega\in \{0,1\}^*$. We consider on $\mathfrak{S}$ the
    orientation induced by the lexicographic ordering on $\Sigma$ by
    the map $\phi_{\tilde \omega}$.\\
    A rational pair $(\frac pq,\frac rq)$ in $\mathfrak{S}$ is
    obtained in the Triangular tree as the mediant of the neighbouring
    pairs, its parents. Then we define \emph{two actions, $L$ and $R$,
      on $(\frac pq,\frac rq)$,} as follows: $(\frac pq,\frac rq)\, L$
    is the rational pair obtained as the mediant of
    $(\frac pq,\frac rq)$ with its left parent and
    $(\frac pq,\frac rq)\, R$ is the rational pair obtained as the
    mediant of $(\frac pq,\frac rq)$ with its right parent.
\end{definition}

\begin{lemma}\label{lemma:boundaryLR}
    Let $\frac ab \in \Q \cap (0,1)$ and let $[a_1,\ldots,a_n]$ be its
    continued fraction expansion. We have
    \[
        \triple{a}{a}{b} =
        \begin{cases}
            \triple{1}{1}{2} L^{a_1-1}R^{a_2}\cdots
            L^{a_{n-1}}R^{a_n-1}\, , & \text{if $n$ is even}\\[0.1cm]
            \triple{1}{1}{2} L^{a_1-1}R^{a_2}\cdots
            R^{a_{n-1}}L^{a_n-1}\,  & \text{if $n$ is odd}
        \end{cases}
    \]
    The same combination of actions sends $\singy{1}{2}{0}$ to
    $\singy{a}{b}{0}$ on $\Lambda$ and $\singx{1}{1}{2}$ to
    $\singx{1}{a}{b}$ on $\Upsilon$.
\end{lemma}
\begin{proof}
    According to~\eqref{m(x)cf}, the path on the Farey tree from
    $\frac 12$ to $\frac ab$ is
    $L^{a_1-1}R^{a_2}\cdots L^{a_{n-1}}R^{a_n-1}$ if $n$ is even and
    $L^{a_1-1}R^{a_2}\cdots R^{a_{n-1}}L^{a_n-1}$ if $n$ is odd. By
    the definition of $L$ and $R$ in Definition
    \ref{def:actionsLR-2d}, this is also the path to reach
    $\triple{a}{a}{b}$ from $\triple{1}{1}{2}$. The same holds on
    $\Lambda$ and $\Upsilon$.
\end{proof}

\begin{remark}
    In the case $n=1$ the above formula reads
    $\triple{a}{a}{b} =\triple{1}{1}{2} L^{a_1-2}$. The same holds
    also in Theorem~\ref{prop:interiorLRI} below.
\end{remark}

The previous lemma shows that the basic moves $L$ and $R$ are enough
to reach every boundary rational pair starting from the midpoints of
the sides of $\triangle$. We now describe how to reach interior
rational pairs along the tree always starting from
$\triple{1}{1}{2}$. Recall that we denote by $\ell$ the open line
segment joining $(1,0)$ and $\couple{1}{2}{1}{2}$. Since interior
rational pairs are located on backward images $\phi_\omega(\ell)$ with
$\omega\in \{0,1\}^*$, our strategy is divided into two steps: first
we describe a path from $\triple{1}{1}{2}$ to the mediant between the
endpoints of $\phi_\omega(\ell)$, and then we encode the sequence of
moves along $\phi_\omega(\ell)$ to reach the considered rational pair.

\begin{definition}
    Let $\omega\in \{0,1\}^*$ with $n=|\omega|\ge 0$ and consider the
    triangle $\triangle_\omega=\phi_\omega(\trianglecl)$, partitioned
    by $\phi_\omega(\ell)$ into two subtriangles. The \emph{action of
      the symbol $I$} on $\phi_\omega\triple{1}{1}{2}$, which we
    denote as a right action by $\phi_\omega\triple{1}{1}{2}I$, gives
    the mediant between $\phi_\omega\triple{1}{1}{2}$ and
    $\phi_\omega(1,0)$. In other words,
    \[
        \phi_\omega\triple{1}{1}{2}I\coloneqq
        \phi_\omega\triple{1}{1}{2}\oplus \phi_\omega(1,0) =
        \phi_\omega\triple{2}{1}{3}.
    \]
\end{definition}

We now give a geometric interpretation of this definition, also
clarified by Figure~\ref{fig:I-action}. The point
$\phi_\omega\triple{1}{1}{2}$ is the right endpoint of
$\phi_\omega(\ell)$ and the action of $I$ is to step from
$\phi_\omega\triple{1}{1}{2}$ to the mediant between the two endpoints
of $\phi_\omega(\ell)$, which is one of the children of
$\phi_\omega\triple{1}{1}{2}$.

\begin{figure}[h]
    \begin{tikzpicture}[scale=1.5]

    \draw[thick] (0,0) node[below left]{$\phi_\omega(0,0)$} -- (4,-1)
    node[below right]{$\phi_\omega(1,0)$} -- (1.9,1.9)
    node[above]{$\phi_\omega(1,1)$} -- cycle;

    \draw[thick] (1,1)
    node[above left]{$\phi_\omega\left(\frac 12, \frac 12\right)$} -- (4,-1);

    \node[gray] at (2,1/3) {\huge $\bullet$};

    \node[below left] at (2,1/3)
    {$\phi_\omega\left(\frac 12, \frac 12\right)I$};

\end{tikzpicture}
    \caption{Right action of the symbol $I$.}\label{fig:I-action}
    \end{figure}
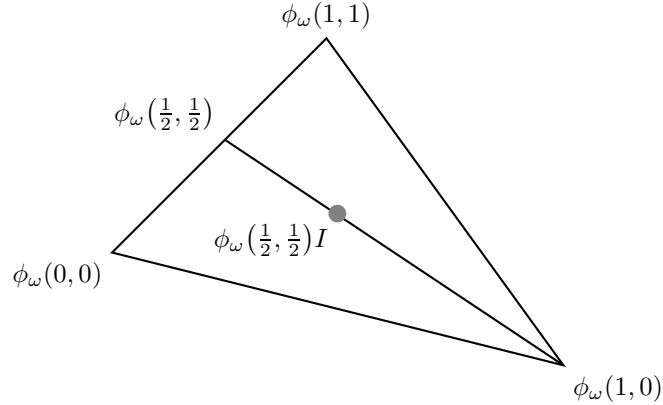

\begin{lemma}\label{lemma:interiorLI}
    Let $n\geq 0$ and let $\omega\in \{0,1\}^*$ with
    $|\omega|=n$. Then
    \[
        \phi_\omega\triple{1}{1}{2} = \triple{1}{1}{2}
        \prod_{i=1}^n W_i,
    \]
    where $W_i=I$ if $\omega_i=0$ and $W_i=L$ if $\omega_i=1$.
\end{lemma}
\begin{proof}
    We argue by induction on $n\geq 0$. The conclusion is trivial when
    $n=0$. Suppose that the thesis holds for all the finite binary
    words of length $n$ and let $\omega$ be a word with
    $|\omega|=n+1$.  Thus $\omega=\tilde\omega 0$ or
    $\omega=\tilde\omega 1$ for some word $\tilde\omega$ of length
    $n$. By definition of $I$ we have
    $\phi_{\tilde\omega}\triple{1}{1}{2}I =
    \phi_{\tilde\omega}\triple{2}{1}{3} = \phi_{\tilde\omega
      0}\triple{1}{1}{2}$, thus the thesis holds for $\tilde\omega
    0$. For $\tilde\omega 1$ note that the left and right parents of
    the point $\phi_\omega\triple{1}{1}{2}$ are $\phi_\omega(0,0)$ and
    $\phi_\omega(1,1)$, respectively. Hence
    \[
        \phi_{\tilde\omega}\triple{1}{1}{2}L=
        \phi_{\tilde\omega}\triple{1}{1}{2}\oplus
        \phi_{\tilde\omega}(0,0) = \phi_{\tilde\omega}\triple{1}{1}{3}
        = \phi_{\tilde\omega 1}\triple{1}{1}{2},
    \]
    so that the thesis is also true for $\tilde\omega 1$.
\end{proof}

\begin{theorem}\label{prop:interiorLRI}
    Let $(\frac pq,\frac rq)$ be the interior rational pair with
    representation $([\alpha_0,\ldots,\alpha_k],[a_1,\ldots,a_n])$
    (see Definition~\ref{repres-interior}) and let
    $\omega=1^{\alpha_0}0\cdots 1^{\alpha_{k-1}}01^{\alpha_k-1}$, so
    that $(\frac pq,\frac rq)\in \phi_\omega(\ell)$. Then:
    \begin{enumerate}
      \item
        $\phi_\omega\triple{1}{1}{2} =
        \triple{1}{1}{2} L^{\alpha_0}I\cdots
        L^{\alpha_{k-1}}IL^{\alpha_k-1}$;
      \item
        $\triple{p}{r}{q} =
        \begin{cases}
            \phi_\omega\triple{1}{1}{2} IL^{a_1-1}R^{a_2}\cdots
            L^{a_{n-1}}R^{a_n-1}\,  & \text{if $n$ is even}\\[0.1cm]
            \phi_\omega\triple{1}{1}{2} IL^{a_1-1}R^{a_2}\cdots
            R^{a_{n-1}}L^{a_n-1}\, , & \text{if $n$ is odd}
        \end{cases}$.
    \end{enumerate}
    Hence
    \[
        \triple{p}{r}{q} =
        \begin{cases}
            \triple{1}{1}{2} L^{\alpha_0}I\cdots
            L^{\alpha_{k-1}}IL^{\alpha_k-1}IL^{a_1-1}R^{a_2}\cdots
            L^{a_{n-1}}R^{a_n-1}\, , & \text{if $n$ is even}\\[0.1cm]
            \triple{1}{1}{2} L^{\alpha_0}I\cdots
            L^{\alpha_{k-1}}IL^{\alpha_k-1}IL^{a_1-1}R^{a_2}\cdots
            R^{a_{n-1}}L^{a_n-1}\, , & \text{if $n$ is odd}
        \end{cases}
    \]
\end{theorem}
\begin{proof}
    (i) It is a straightforward consequence of Lemma~\ref{lemma:interiorLI}.\\[0.2cm]
    (ii) Let $\frac ab$ the unique rational number associated to
    $(\frac pq,\frac rq)$ according to Lemma~\ref{lemma:key}-(ii), so
    that by definition $[a_1,\ldots,a_n]$ is its continued fraction
    expansion. Thus $\singx{1}{a}{b}$ can be reached from
    $\singx{1}{1}{2}$ with the sequence
    $L^{a_1-1}R^{a_2}\cdots L^{a_{n-1}}R^{a_n-1}$ if $n$ is even or
    with $L^{a_1-1}R^{a_2}\cdots R^{a_{n-1}}L^{a_n-1}$ if $n$ is odd,
    from Lemma~\ref{lemma:boundaryLR}. Since the maps $\phi_0$ and
    $\phi_1$ commute with the mediant operation, it is easy to prove
    that also the parent-child relationship is preserved. Thus by the
    above sequence of $L$ and $R$ moves one obtains
    $(\frac pq,\frac rq)$ from
    $\phi_{\omega 1}\singx{1}{1}{2} = \phi_\omega \triple{2}{1}{3}
    =\phi_\omega\triple{1}{1}{2} I$, which is the mediant between the
    endpoints of $\phi_\omega(\ell)$.
\end{proof}

To each interior rational pair we have thus associated a finite word
over the alphabet $\{L,R,I\}$, with the geometric meaning of telling
how to move on the Triangular tree to reach the rational pair under
consideration starting from the root $\triple{1}{1}{2}$. In
particular, this finite word is the concatenation between a word over
$\{L,I\}$ ending necessarily with the symbol $I$ and a word over
$\{L,R\}$: the first word gives the path to reach
$\phi_\omega\triple{1}{1}{2} I$, which is the rational point on
$\phi_\omega(\ell)$ appearing on the level of the tree with smallest
index and also the mediant between its endpoints; and the second word
gives the moves along the line segment $\phi_\omega(\ell)$ to reach
the given point from the mediant of its endpoints.

\vskip 0.2cm To continue the analogy with the one-dimensional case, we
now convert the actions of the symbols $L$, $R$, and $I$, into the
actions by right multiplication of three $SL(3,\Z)$ matrices, using a
matrix representation for the rational pairs in $\trianglecl$. Towards
this aim we recall the correspondence between rational pairs and
three-dimensional vectors we have introduced in Section
\ref{sec:triangle-seq}.

\begin{definition}\label{def:matrix}
    Let $(\frac pq,\frac rq)$ be a rational pair in
    $\trianglecl\setminus\mathcal{T}_{-1}$. We associate to
    $(\frac pq,\frac rq)$ the $3\times3$ matrix $\mattt{p}{r}{q}$ defined
    as follows. The first two columns are respectively the right and
    the left parents of $(\frac pq,\frac rq)$, expressed as
    three-dimensional vectors. The third column depends on the
    location of the point:
    \begin{enumerate}
      \item if the pair is a boundary point, the third column is the
        vertex of $\trianglecl$ which is opposite to the side containing
        the pair;
      \item if $(\frac pq,\frac rq)\in \phi_\omega(\ell)$ for
        $\omega\in\{0,1\}^*$, the third column is the vertex ``2'' of
        $\triangle_{\omega}$, that is $\phi_{\omega}(1,1)$.
    \end{enumerate}
    The three-dimensional vector associated to $(\frac pq,\frac rq)$
     is obtained as the sum of the first
    two columns of its matrix.
\end{definition}

Since they will take on great importance, we explicitly show the
matrices representing the midpoints of the three sides of $\triangle$:
    \[
        \mattt{1}{1}{2} =
    \begin{pmatrix}
        1&1&1\\1&0&1\\1&0&0
    \end{pmatrix},
    \quad
    \mattty{1}{2}{0} =
    \begin{pmatrix}
        1&1&1\\1&0&1\\0&0&1
    \end{pmatrix},
    \quad
    \matttx{1}{1}{2} =
    \begin{pmatrix}
        1&1&1\\1&1&0\\1&0&0
    \end{pmatrix}.
\]
Let us now define
\begin{equation} \label{le-tre-matrici}
    L\coloneqq\begin{pmatrix} 1&0&0\\1&1&0\\0&0&1\end{pmatrix},\quad
    R\coloneqq \begin{pmatrix} 1&1&0\\0&1&0\\0&0&1\end{pmatrix},\quad
    I\coloneqq\begin{pmatrix} 1&0&1\\1&0&0\\0&1&0\end{pmatrix}.
\end{equation}
Note that all the above matrices are in $SL(3,\Z)$ and that $L$ and
$R$ extend their $2\times 2$ counterpart defined in
\eqref{le-due-matrici}. A straightforward computation shows that the
action of each of the above matrices by right multiplication on the
matrix representing a rational pair yields the matrix of the child
obtained according to the move bearing the same name of the matrix. As
a direct consequence of Lemma~\ref{lemma:boundaryLR} and
Theorem~\ref{prop:interiorLRI}, we thus have the following properties.
\begin{enumerate}[label={\upshape(\arabic*)},wide = 0pt,labelsep=0.2cm,leftmargin=*]
  \item Let $\frac ab \in \Q\cap (0,1)$ and let $[a_1,\,\ldots,\,a_n]$
    be its continued fraction expansion. Then
    \[
        \mattt{a}{a}{b} =
        \begin{cases}
            \mattt{1}{1}{2} L^{a_1-1}R^{a_2}\cdots
            L^{a_{n-1}}R^{a_n-1}\, , & \text{if $n$ is even}\\[0.1cm]
            \mattt{1}{1}{2} L^{a_1-1}R^{a_2}\cdots
            R^{a_{n-1}}L^{a_n-1}\, , & \text{if $n$ is odd}
        \end{cases}
    \]
    The same right action yields $\mattty{a}{b}{0}$ from
    $\mattty{1}{2}{0}$ and $\matttx{1}{a}{b}$ from $\matttx{1}{1}{2}$.
  \item Let $(\frac pq,\frac rq)$ be the interior rational pair with
    representation
    $([\alpha_0,\ldots,\alpha_k],[a_1,\ldots,a_n])$. Then
    \[
        \mattt{p}{r}{q} =
        \begin{cases}
            \mattt{1}{1}{2} L^{\alpha_0}I\cdots
            L^{\alpha_{k-1}}IL^{\alpha_k-1}IL^{a_1-1}R^{a_2}\cdots
            L^{a_{n-1}}R^{a_n-1}\, , & \text{if $n$ is even}\\[0.1cm]
            \mattt{1}{1}{2} L^{\alpha_0}I\cdots
            L^{\alpha_{k-1}}IL^{\alpha_k-1}IL^{a_1-1}R^{a_2}\cdots
            R^{a_{n-1}}L^{a_n-1}\, , & \text{if $n$ is odd}
        \end{cases}
    \]
\end{enumerate}

\begin{example}\label{ex:pathLRI}
    Consider the point in with representation
    $([2,0,1,1],[2,2])$. Using
    $\ell=\phi_1\phi_0\phi_2(\Sigma\setminus\set{(0,0),(1,1)})$ and
    \eqref{eq:backSigma}, our point is located along the open line
    segment $\phi_{\omega}(\ell)$ with $\omega=110010$, whose left and
    right endpoints can be readily computed and are respectively:
    \[
        \phi_{110010}(1,0) = \triple{3}{2}{8}
        \quad\text{and}\quad
        \phi_{110010}\triple{1}{1}{2} = \triple{5}{4}{15}.
    \]
    By Theorem~\ref{prop:interiorLRI} we have that the path along
    the Triangular tree to reach this point from the root
    $\triple{1}{1}{2}$ is
    \[
        L^2IILIILR,
    \]
    which is the concatenation of the words $L^2IILII$ and $LR$. The
    word $L^2IILII$ brings $\triple{1}{1}{2}$ first to
    $\triple{1}{1}{2}L^2IILI = \triple{5}{4}{15}$, the right endpoint
    of $\phi_{110010}(\ell)$, and then to
    $\triple{1}{1}{2}L^2IILII = \triple{5}{4}{15}I =
    \triple{8}{6}{23}$, the mediant between the endpoints of
    $\phi_{110010}(\ell)$. The word $LR$ shows how to move along
    $\phi_{110010}(\ell)$ starting from $\triple{8}{6}{23}$ to obtain
    our point:
    \[
        \triple{8}{6}{23}LR =
        \triple{11}{8}{31}R = \triple{19}{14}{54}.
    \]
    Figure~\ref{fig:example} shows the the path in the triangle to
    reach $\phi_{110010}(\ell)$ and the moves along this line segment
    to get to our point. The corresponding matrix multiplication
    yields
    \[
        \mattt{1}{1}{2}L^2IILIILR =
        \begin{pmatrix}
            23&31&11\\8&11&4\\6&8&3
        \end{pmatrix},
    \]
    which is by definition the matrix representing
    $\triple{19}{14}{54}$.
\end{example}

\begin{figure}[h]
    \begin{tikzpicture}[scale=12,spy using outlines={circle, magnification=2.3, connect spies}]

    \def\o{-0.0005}
    \def\v{0.002}
	
	\def\x{3/5}
	\def\y{1}
	
    \draw[very thick] (0,0) -- (1,0) -- (1/2,1/2) -- cycle;

    \node[above,yshift=5pt] at (1/2,1/2) {$\left(\frac 12,\frac 12\right)$};
    \node[below] at (0,0) {$(0,0)$};
    \node[below] at (1,0) {$(1,0)$};

    \node[above,yshift=0pt,xshift=-10pt] at (1/3,1/3) {$\left(\frac 13,\frac 13\right)$};
    \node[above,yshift=-5pt,xshift=-15pt] at (1/4,1/4) {$\left(\frac 14,\frac 14\right)$};
    \node[below] at (2/5,1/5-\v) {\tiny $\left(\frac 25,\frac 15\right)$};
    \node[blue,right,yshift=-1pt] at (3/8,2/8) {\tiny $\left(\frac 38,\frac 28\right)$};
    \node[right,yshift=6pt] at (4/11,3/11) {\tiny $\left(\frac 4{11},\frac 3{11}\right)$};
    \node[blue,above,xshift=-14pt] at (5/15,4/15) {\tiny $\left(\frac 5{15},\frac 4{15}\right)$};

    \draw (1,0) -- (1/2,1/2);
    \draw (1,0) -- (1/3,1/3);
    \draw (1,0) -- (1/4,1/4);
    \draw (1/3,1/3) -- (2/5,1/5);
    \draw (3/8,2/8) -- (1/4,1/4);
    \draw (4/11,3/11) -- (1/4,1/4);
    \draw[blue,thick] (5/15,4/15) -- (3/8,2/8);

    \node[gray] at (1/2-\o,1/2-\v) {\large $\bullet$};
    \node[gray] at (1/3-\o,1/3-\v) {\large $\bullet$};
    \node[gray] at (1/4-\o,1/4-\v) {\large $\bullet$};

    \node[gray] at  (2/5,1/5) {\large $\bullet$};
    \node[blue] at (3/8,2/8) {\large $\bullet$};
    \node[gray] at (4/11,3/11) {\large $\bullet$};
    \node[blue] at (5/15,4/15) {\large $\bullet$};

    \coordinate (spypoint) at (5/15,4/15);
    \coordinate (magnifyglass) at (1/5,-1/3);
    \spy [gray, size=7cm] on (spypoint) in node[fill=white] at (magnifyglass);
	
	\draw[blue,thick] (\x,-1/3) node{\large $\bullet$} -- (\y,-1/3) node{\large $\bullet$};
	
	\node[blue,below] at (\x,-1/3) {$\left(\frac 38,\frac 28\right)$};
	\node[blue,below] at (\y,-1/3) {$\left(\frac 5{15},\frac 4{15}\right)$};
	
	\node[below,yshift=-10pt] (A) at  ({(\x + \y)/2},-1/3) {$\left(\frac 8{23},\frac {6}{23}\right)$};
	\node[below,yshift=-30pt] at ({(\y+2*\x)/3},-1/3) {$\left(\frac {11}{31},\frac {8}{31}\right)$};
	\node[above,yshift=20pt] at ({(2*\y+3*\x)/5},-1/3) {$\left(\frac {19}{54},\frac {14}{54}\right)$};

	\draw[densely dotted] ({(\x + \y)/2},-1/3) -- ({(\x + \y)/2},-1/3-10pt);
	\draw[densely dotted] ({(\y+2*\x)/3},-1/3) -- ({(\y+2*\x)/3},-1/3-1/12);
	\draw[densely dotted] ({(2*\y+3*\x)/5},-1/3) -- ({(2*\y+3*\x)/5},-1/3+1/18);
	
	\node[gray] at  ({(\x + \y)/2},-1/3) {\large $\bullet$};
	\node[gray] at  ({(\y+2*\x)/3},-1/3) {\large $\bullet$};
	\node[gray] at  ({(2*\y+3*\x)/5},-1/3) {\large $\bullet$};
	
	\node[blue,above,yshift=3pt] at (\y,-1/3) {$\phi_{110010}(\ell)$};
\end{tikzpicture}
    \caption{Path to reach the right endpoint $\triple{5}{4}{15}$ of
      the line segment $\phi_{110010}(\ell)$ with the moves
      $L^2IILI$. The line segment $\phi_{110010}(\ell)$ is enhanced in
      blue and the remaining moves $ILR$ are represented in the
      bottom-right figure.}\label{fig:example}
\end{figure}
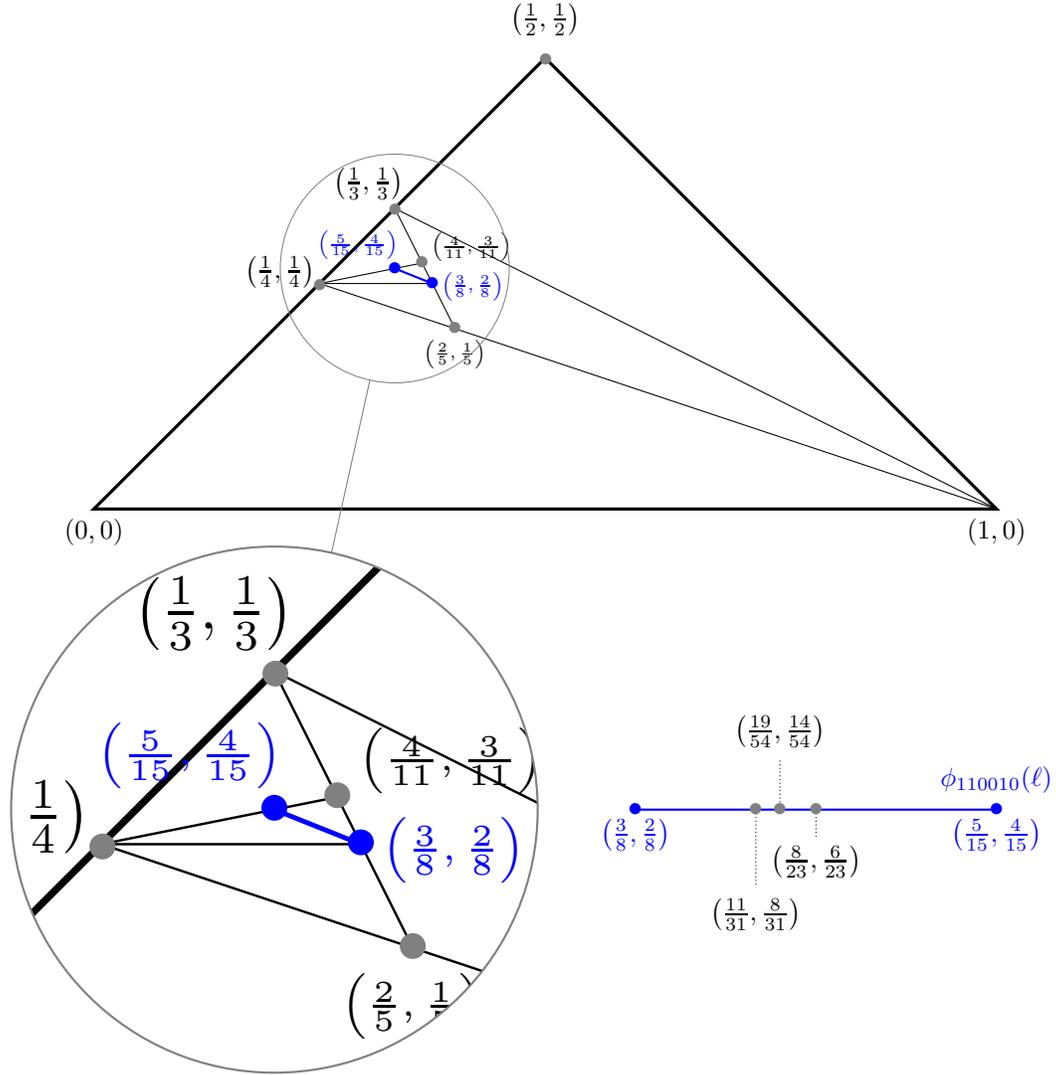

\vskip 0.2cm The last step to complete the similarity with the
one-dimensional case is a way to use the coding we have introduced to
express a rational pair as a counterimage of $(\frac 12,\frac 12)$ and
vice versa. The key observation is that using the local inverses
$\phi_0$, $\phi_1$, and $\phi_2$, it is possible to mimic the
behaviour of the one-dimensional Farey map along $\Sigma$ in terms of
the inverses $\psi_0$ and $\psi_1$ defined in
\eqref{farey-inversi}. In fact a straightforward computation shows
that for $0\leq x\leq 1$ it holds
\begin{equation}\label{farey-diag}
  \phi_1(x,x) = (\psi_0(x),\psi_0(x))
  \quad\text{and}\quad
  \phi_0^2\phi_2(x,x)=(\psi_1(x),\psi_1(x)).
\end{equation}

\begin{proposition}\label{prop:backimage}
    \begin{enumerate}
      \item Let $\frac ab$ be a rational number with continued
        fraction expansion $[a_1,\,\ldots,\,a_n]$. Then
        \begin{align*}
            \triple{a}{a}{b} &=\phi_1^{a_1-1}\phi_0^2\phi_2 \cdots
            \phi_1^{a_{n-1}-1}\phi_0^2\phi_2
            \phi_1^{a_n-2}\triple{1}{1}{2},\\
            \singy{a}{b}{0} &= \phi_2\circ
            \phi_1^{a_1-1}\phi_0^2\phi_2 \cdots
            \phi_1^{a_{n-1}-1}\phi_0^2\phi_2
            \phi_1^{a_n-2}\triple{1}{1}{2},\\
            \singx{1}{a}{b} &= \phi_0\circ \phi_2\circ
            \phi_1^{a_1-1}\phi_0^2\phi_2
            \cdots\phi_1^{a_{n-1}-1}\phi_0^2\phi_2
            \phi_1^{a_n-2}\triple{1}{1}{2}.
        \end{align*}
      \item Let $(\frac pq,\frac rq)$ be the interior rational pair with
        representation
        $([\alpha_0,\ldots,\alpha_k],[a_1,\ldots,a_n])$.
        Then
        \[
            \triple{p}{r}{q}=\phi_1^{\alpha_0}\phi_0 \cdots
            \phi_1^{\alpha_k}\phi_0\circ\phi_2\circ\phi_1^{a_1-1}\phi_0^2\phi_2
            \cdots\phi_1^{a_{n-1}-1}\phi_0^2\phi_2
            \phi_1^{a_n-2}\triple{1}{1}{2}.
        \]
    \end{enumerate}
\end{proposition}
\begin{proof}
    (i) From the properties of the Farey map we have
    \[
        \frac ab = \psi_0^{a_1-1}\psi_1\cdots
        \psi_0^{a_n-2}\left(\frac12\right).
    \]
    By using the correspondence between $\psi_0$ and $\phi_1$ and
    between $\psi_1$ and $\phi_0^2\phi_2$ we can conclude for
    $\triple{a}{a}{b}$. The expressions for $\singy{a}{b}{0}$ and
    $\singx{1}{a}{b}$ are now a trivial consequence.\\[0.2cm]
    (ii) Let $\frac ab$ the unique rational number associated to
    $(\frac pq,\frac rq)$ according to Lemma~\ref{lemma:key}-(i), so that
    by definition $[a_1,\ldots,a_n]$ is its continued fraction
    expansion. Equation~\eqref{eq:backSigma} gives
    $(\frac pq,\frac rq)=\phi_1^{\alpha_0}\phi_0 \cdots
    \phi_1^{\alpha_k}\phi_0\phi_2\triple{a}{a}{b}$, so that the
    conclusion holds directly from (i).
\end{proof}

\begin{remark}
    (i) In the statement of the above proposition we have explicitly used
    the composition sign to emphasise the separation between the first
    group of maps, the map $\phi_2$, and the second group, because
    they correspond to the different parts of
    the representation.\\[0.2cm]
    (ii) Note that the map $\phi_2$ separating the two groups of
    maps is an actual separator, meaning that if we exhibit a point as
    a counterimage of $\triple{1}{1}{2}$, it can be unambiguously
    identified. In fact it is the only map $\phi_2$ which is preceded
    by a single occurrence of $\phi_0$.
\end{remark}

Also the definition of rank can be extended to our two-dimensional
setting. We say that $\rank(\frac pq,\frac rq)=m$ if and only if
$x\in \mathcal{T}_{m}\setminus\mathcal{T}_{m-1}$.

\begin{corollary}
    Let $(\frac pq,\frac rq)$ be the rational pair with representation
    $([\alpha_0,\ldots,\alpha_k],[a_1,\ldots,a_n])$. Then
    \[
        \rank\triple{p}{r}{q} =
        \begin{cases}
            \sum_{j=0}^k \alpha_j + \sum_{i=1}^n a_i +k-2\, , &
            \text{if the triangle sequence exists and $\alpha_k>0$}\\[0.1cm]
            \sum_{i=1}^n a_i-2\, , & \text{otherwise}
        \end{cases}
    \]
\end{corollary}
\begin{proof}
    By definition, $\rank\triple{1}{1}{2}=0$. Consider the expression
    of a rational pair as a backward image of $\triple{1}{1}{2}$ given
    in Proposition~\ref{prop:backimage} and look separately to its (at
    most) three blocks of compositions, starting from the right to the
    left.
    \begin{enumerate}[label={(\arabic*)},wide=0pt,labelsep=0.2cm,leftmargin=*]
      \item In the first group of maps, the one corresponding to the
        continued fraction expansion, each application of $\phi_1$ and
        each block $\phi_0^2\phi_2$ increases the rank by 1
        ($\phi_0^2\phi_2$ does not increase the rank more because it
        is applied to points on $\Sigma$). Thus the rank of
        $\phi_1^{a_1-1}\phi_0^2\phi_2
        \cdots\phi_1^{a_{n-1}-1}\phi_0^2\phi_2
        \phi_1^{a_n-2}\triple{1}{1}{2}$ is $\sum_{i=1}^n a_i-2$.
      \item The application of the separator $\phi_2$ appears if and
        only if the point is not along $\Sigma$ and anyway does not
        change the level.
      \item The third block appears if and only if the point is along
        $\Upsilon$ or the point is interior. Moreover, when the third
        block appears, its right end is exactly one application of
        the map $\phi_0$, but since it acts on a point on $\Lambda$, it
        does not change the level. Each of the subsequent maps increases
        the level by $1$, for a total increase of
        $\sum_{j=0}^k \alpha_j+k$. Thus note that the block related to
        the triangle sequence contributes to the rank if and only if
        the point is an interior rational pair.
    \end{enumerate}
\end{proof}

\begin{example}
    As in Example~\ref{ex:pathLRI}, consider the point
    $\triple{19}{14}{54}$ whose representation is
    $([2,0,1,1],[2,2])$. By Proposition~\ref{prop:backimage}, we have
    \[
        \triple{19}{14}{54} = \phi_1^2\phi_0\phi_0\phi_1\phi_0\phi_1\phi_0\circ \phi_2\circ
        \phi_1\phi_0^2\phi_2\triple{1}{1}{2} =
        \phi_1^2\phi_0\phi_0\phi_1\phi_0\phi_1\phi_0\circ
        \phi_2\triple{2}{2}{5},
    \]
    where we have emphasised that our point is a backward image of
    $\triple{2}{2}{5}$, which corresponds to the second component of
    the representation since $\frac 25=[2,2]$.
\end{example}


\section{Approximations for non-rational pairs} \label{sec:approx}

In this section we use the structure of the Triangular tree to define
approximations of real pairs in $\trianglecl$ by rational pairs. This
is particularly important for non-rational pairs, which have an
infinite number of rational approximations.

Non-rational pairs have either an infinite triangle sequence or a
finite triangle sequence with infinite associated continued fraction
expansion. An analogous of Proposition~\ref{prop:backimage} holds also
for non-rational pairs, apart from the non-convergent infinite case,
for which Proposition~\ref{prop:nonconvergent} basically implies that
such a result is not possible.

\subsection{Finite triangle sequence} Let us first consider the case
of real pairs $(\alpha,\beta)$ with representation of the form
$([\alpha_0,\,\ldots,\,\alpha_k],[a_1,\,a_2,\,\ldots])$ (the second
component being finite or infinite). By extending the ideas of
Section~\ref{sec:rationals}, we associate to a non-rational pair
$(\alpha,\beta)$ the infinite word $\mathcal{W}$ over the alphabet
$\{L,R,I\}$ defined to be
\[
    \mathcal{W}(\alpha,\beta) = L^{\alpha_0}I\cdots
    L^{\alpha_{k-1}}IL^{\alpha_k-1}IL^{a_1-1}R^{a_2}L^{a_3}\cdots.
\]
Moreover we define the \emph{approximations} of $(\alpha,\beta)$ to be
the rational pairs $(\frac{m_j}{s_j},\frac{n_j}{s_j})$ with coding
$\mathcal{W}_{(j)}$, the prefix of $\mathcal{W}$ of length $j$.

\begin{example}
    We consider the rational pair $\triple{19}{14}{54}$ with
    representation $([2,0,1,1],[2,2])$. As shown in
    Example~\ref{ex:pathLRI} the word associated to this point is
    $\mathcal{W} = LLIILIILR$. Thus its approximations are the
    following:
    \[
        \begin{array}{cll}
            \triple{1}{1}{2}    \quad &\mathcal{W}_{(0)} = \eps      \quad &([1],[2])\\[0.2cm]
            \triple{1}{1}{3}    \quad &\mathcal{W}_{(1)} = L         \quad &([2],[3])\\[0.2cm]
            \triple{1}{1}{4}    \quad &\mathcal{W}_{(2)} = LL        \quad &([3],[4])\\[0.2cm]
            \triple{2}{1}{5}    \quad &\mathcal{W}_{(3)} = LLI       \quad &([3],[2])\\[0.2cm]
            \triple{3}{2}{8}    \quad &\mathcal{W}_{(4)} = LLII      \quad &([2,1],[2])\\[0.2cm]
            \triple{4}{3}{11}   \quad &\mathcal{W}_{(5)} = LLIIL     \quad &([2,1],[2])\\[0.2cm]
            \triple{5}{4}{15}   \quad &\mathcal{W}_{(6)} = LLIILI    \quad &([2,0,2],[2])\\[0.2cm]
            \triple{8}{6}{23}   \quad &\mathcal{W}_{(7)} = LLIILII   \quad &([2,0,1,1],[2])\\[0.2cm]
            \triple{11}{8}{31}  \quad &\mathcal{W}_{(8)} = LLIILIIL  \quad &([2,0,1,1],[3])\\[0.2cm]
            \triple{19}{14}{54} \quad &\mathcal{W}_{(9)} = LLIILIILR \quad &([2,0,1,1],[2,2])\\[0.2cm]
        \end{array}
    \]
\end{example}

\begin{example}
    We now consider the non-rational pair
    $(\alpha,\beta) = \singy{1}{2}{\sqrt{2}-1}$, with triangle
    sequence $[1,1]$. Since
    $T^2(\alpha,\beta)=\left(\frac{\sqrt{2}-1}2,0\right)$ and
    $\frac{\sqrt{2}-1}2 = [\widebar{4,1}]$, we have that the
    representation of $(\alpha,\beta)$ is
    $\left([1,1],[\widebar{4,1}]\right)$ and thus the infinite word
    associated to $(\alpha,\beta)$ is
    \[
        \mathcal{W} =LIIL^3\widebar{RL^4}.
    \]
    Thus its first ten approximations are the following:
    \[
        \begin{array}{cll}
            \triple{1}{1}{2}    \quad &\mathcal{W}_{(0)} = \eps      \quad &([1],[2])\\[0.2cm]
            \triple{1}{1}{3}    \quad &\mathcal{W}_{(1)} = L         \quad &([2],[3])\\[0.2cm]
            \triple{2}{1}{4}    \quad &\mathcal{W}_{(2)} = LI        \quad &([2],[2])\\[0.2cm]
            \triple{3}{2}{6}    \quad &\mathcal{W}_{(3)} = LII       \quad &([1,1],[2])\\[0.2cm]
            \triple{4}{3}{8}    \quad &\mathcal{W}_{(4)} = LIIL      \quad &([1,1],[3])\\[0.2cm]
            \triple{5}{4}{10}   \quad &\mathcal{W}_{(5)} = LIILL     \quad &([1,1],[4])\\[0.2cm]
            \triple{6}{5}{12}   \quad &\mathcal{W}_{(6)} = LIILLL    \quad &([1,1],[5])\\[0.2cm]
            \triple{11}{9}{22}  \quad &\mathcal{W}_{(7)} = LIILLLR   \quad &([1,1],[4,2])\\[0.2cm]
            \triple{17}{14}{34} \quad &\mathcal{W}_{(8)} = LIILLLRL  \quad &([1,1],[4,1,2])\\[0.2cm]
            \triple{23}{19}{46} \quad &\mathcal{W}_{(9)} = LIILLLRLL \quad &([1,1],[4,1,3])\\[0.2cm]
            \vdots&&
        \end{array}
    \]
\end{example}

Since the triangle sequence is finite, if $(\alpha,\beta)$ is a
non-rational pair, after a finite transient the approximations have
the same triangle sequence as $(\alpha,\beta)$, and they are simply
given by truncating the continued fraction expansion in the second
component of the representation of $(\alpha,\beta)$. We are using the
fact that
\[
           \lim_{n\rightarrow \infty} \phi_1^{\alpha_0}\phi_0 \cdots
            \phi_1^{\alpha_k}\phi_0\circ\phi_2\circ\phi_1^{a_1-1}\phi_0^2\phi_2
            \cdots\phi_1^{a_{n-1}-1}\phi_0^2\phi_2
            \phi_1^{a_n-2}\triple{1}{1}{2} = (\alpha,\beta)\, ,
\]
which follows by Proposition \ref{prop:backimage} and the properties of the continued fraction expansion translated into our framework with \eqref{farey-diag}.

By learning from the previous examples we can give a general way of
constructing the coding associated to the approximations of a real
pair $(\alpha,\beta)$ with representation
$([\alpha_0,\,\ldots,\,\alpha_k],[a_1,\,a_2,\,\ldots])$. The
correspondence between the words and the representation of the related
approximation can be recovered by the following properties, which
easily follows from Theorem~\ref{prop:interiorLRI} and Proposition
\ref{prop:backimage}.
\begin{enumerate}[label={\upshape(\arabic*)},wide=0pt,labelsep=0.2cm,leftmargin=*]
  \item For $j$ a non-negative integer, the word $L^j$ corresponds to
    the representation $([j+1],[j+2])$.
  \item For $j$, $r$, and $d_0,\,\ldots,\,d_r$ non-negative integers,
    the word $L^{d_0}I\cdots L^{d_r}IL^j$ corresponds to the
    representation $([d_0,\,\ldots,\,d_{r-1},d_r+1],[j+2])$.
  \item For $r$ and $s$ non-negative integers, and for
    $d_0,\ldots,d_r$, and $c_0,\,\ldots,\,c_s$ non-negative integers,
    the word $L^{d_0}I\cdots L^{d_r}IL^{c_0}R^{c_1}\cdots L^{c_s}$ if
    $s$ is even (or the word
    $L^{d_0}I\cdots L^{d_r}IL^{c_0}R^{c_1}\cdots R^{c_s}$ if $s$ is
    odd) corresponds to the representation
    $([d_0,\,\ldots,\,d_{r-1},d_r+1],[c_0+1,\,c_1,\,\ldots,\,c_{s-1},\,c_s+1])$.
\end{enumerate}

\subsection{Convergent infinite triangle sequence}
If the real pair $(\alpha,\beta)$ has a convergent infinite triangle
sequence $[\alpha_0,\,\alpha_1,\,\alpha_2,\,\ldots]$, its
representation is $([\alpha_0,\,\alpha_1,\,\alpha_2\,\ldots],[2])$ and
\eqref{converg-ok} holds. Thus we associate to $(\alpha,\beta)$ the
infinite word $\mathcal{W}$ over the alphabet $\{L,R,I\}$ defined by
\[
\mathcal{W}(\alpha,\beta) =  L^{\alpha_0}IL^{\alpha_1}IL^{\alpha_2}I\cdots\, ,
\]
and the \emph{approximations} of $(\alpha,\beta)$ are again the
rational pairs $(\frac{m_j}{s_j},\frac{n_j}{s_j})$ with coding
$\mathcal{W}_{(j)}$, the prefix of $\mathcal{W}$ of length $j$. Rules
(1) and (2) above still hold and show how to construct the
representations of the approximations. Notice that the choice of $[2]$
as the second component of the representation is arbitrary, and $[2]$
can be replaced by any continued fraction expansion. However
this choice does not change the form of the word $\mathcal{W}$.

\begin{example}
    Let $\alpha_0=1$ and for $k\geq 1$
    let $\alpha_k=p_k$, with $p_k$ being the $k$-th prime number, so
    that the triangle sequence we are considering is
    $[1,\,2,\,3,\,5,\,7,\,11,\,\ldots]$. From \eqref{eq:1-lambda}, for
    $k\geq 0$ we have
    \[
        1-\lambda_k = \alpha_{k+1}\frac{q_{k-1}}{q_{k+1}} =
        \frac{\alpha_{k+1}}{\alpha_{k+1} +
          \frac{q_{k-2}+q_k}{q_{k-1}}}\leq \frac{\alpha_{k+1}}{1+\alpha_{k+1}},
    \]
    where the last inequality holds since
    $\frac{q_{k-2}+q_k}{q_{k-1}}\geq \frac{q_k}{q_{k-1}}\geq 1$. Each
    triangle sequence digit is non-zero, thus $\lambda_k<1$ for all
    $k\geq 0$ and we thus have
    \[
        \prod_{k=0}^\infty (1-\lambda_k) = \prod_p \frac{p}{1+p},
    \]
    where the right-hand-side product extends over all the prime
    numbers. We thus have
    \[
        \prod_p \frac{p}{1+p} = \prod_p \left(1-\frac{1}{p+1}\right) <
        \prod_{p>2} \left(1-\frac 1p\right)=0,
    \]
    and the last product diverges to $0$ because the sum of the
    reciprocals of the primes diverges. As a consequence, the given
    triangle sequence represents a unique point of $\trianglecl$. This
    point is represented by the word
    \[
        \mathcal{W} =LIL^2IL^3IL^5IL^7IL^{11}I\cdots
    \]
    and its first approximations are
    \[
        \begin{array}{cll}
            \triple{1}{1}{2}   \quad &\mathcal{W}_{(0)} = \eps      \quad &([1],[2])\\[0.2cm]
            \triple{1}{1}{3}   \quad &\mathcal{W}_{(1)} = L         \quad &([2],[3])\\[0.2cm]
            \triple{2}{1}{4}   \quad &\mathcal{W}_{(2)} = LI        \quad &([2],[2])\\[0.2cm]
            \triple{3}{1}{5}   \quad &\mathcal{W}_{(3)} = LIL       \quad &([2],[3])\\[0.2cm]
            \triple{4}{1}{6}   \quad &\mathcal{W}_{(4)} = LILL      \quad &([2],[4])\\[0.2cm]
            \triple{5}{2}{8}   \quad &\mathcal{W}_{(5)} = LILLI     \quad &([1,3],[2])\\[0.2cm]
            \triple{6}{3}{10}  \quad &\mathcal{W}_{(6)} = LILLIL    \quad &([1,3],[3])\\[0.2cm]
            \triple{7}{4}{12}  \quad &\mathcal{W}_{(7)} = LILLILL   \quad &([1,3],[4])\\[0.2cm]
            \triple{8}{5}{14}  \quad &\mathcal{W}_{(8)} = LILLILLL  \quad &([1,3],[5])\\[0.2cm]
            \triple{11}{6}{19} \quad &\mathcal{W}_{(9)} = LILLILLLI \quad &([1,2,4],[2])\\[0.2cm]
            \vdots&&
        \end{array}
    \]
\end{example}

\subsection{Non-convergent infinite case}
Let $(\alpha,\beta)$ have a non-convergent infinite triangle sequence
$[\alpha_0,\,\alpha_1,\,\dots]$. It means that $(\alpha,\beta)$ is in
the line segment $\mathfrak{L}$ of points sharing the same triangle
sequence. In this case we have seen that the representation
$([\alpha_0,\,\alpha_1,\,\ldots],[2])$ encodes the whole segment, and
it is impossible to distinguish different points on $\mathfrak{L}$ by
using it.

Here we propose a possible way to construct approximations of
$(\alpha,\beta)$ following the methods used in the other cases. This
proposal is certainly not the only meaningful and not the most
``natural'' in any sense. For all $j\ge 0$, the point $(\alpha,\beta)$
is in $\triangle(\alpha_0,\ldots,\alpha_j)$, so that
\[
    (\xi_j,\eta_j) \coloneqq T^j(\alpha,\beta) \in
    \triangle(\alpha_j).
\]
Since $\alpha_j\to +\infty$, the second component $\eta_j$ is
vanishing as $j$ increases. Let then
$\frac{p_j(j)}{q_j(j)} = [a_1(j),\,\ldots,\,a_j(j)]$ be the $j$-th
convergent of $\xi_j$, obtained from the continued fraction expansion
$[a_1(j),\,a_2(j),\,a_3(j),\,\ldots]$ of $\xi_j$. Then for
$\omega=1^{\alpha_0}0\dots1^{\alpha_{j-1}}0$ it holds
\[
    \triple{m_j}{n_j}{s_j} \coloneqq \phi_{\omega}
    \triple{p_j(j)}{0}{q_j(j)} \in \triangle(\alpha_0,\,\ldots,\,\alpha_j)
\]
and
\[
    \lim_{j\to \infty} \Big|(\alpha,\beta) - \triple{m_j}{n_j}{s_j}
    \Big| =0.
\]
We then consider the approximations of $(\alpha,\beta)$ given by the
rational pairs with representations
\[
    ([\alpha_0,\,\ldots,\,\alpha_j],[a_1(j),\,\ldots,\,a_j(j)])
\]
for $j\ge 0$, defined as above.


\section{Speed of the approximations} \label{sec:speed}

In this final section we study the problem of the speed of the
approximations introduced before. In particular given a non-rational
pair $(\alpha,\beta) \in \trianglecl$ we have defined a sequence
$(\frac{m_j}{s_j},\frac{n_j}{s_j})$ of rational pairs in the
Triangular tree which approximate $(\alpha,\beta)$. The speed of the
approximations we consider concerns the supremum of the exponents
$\eta>0$ for which
\begin{equation}
    \label{prob-speed}
    \liminf_{j \to \infty} s_j^\eta
    \left|\alpha-\frac{m_j}{s_j}\right|
    \left|\beta-\frac{n_j}{s_j}\right| = 0.
\end{equation}
This problem is also known as the problem of simultaneous
approximations of real numbers, and two famous problems in this
research area are Dirichlet's Theorem and Littlewood's Conjecture (see
\cite{schmidt-book} and \cite{dioph-review} for more details).

We start with some notation. Let $(\alpha,\beta)$ be a point in
$\trianglecl$ with $\alpha$ or $\beta$ irrational. Let
$\omega \in \set{0,1}^*$ and $\phi_{\omega}$ the correspondent
concatenation of $\phi_0$ and $\phi_1$, we use the notation
\begin{equation}\label{denom-multi}
    \phi_{\omega}\triple{p}{r}{q} = \left(
      \frac{m_{\omega}(p,r,q)}{s_{\omega}(p,r,q)} ,
      \frac{n_{\omega}(p,r,q)}{s_{\omega}(p,r,q)} \right) ,
\end{equation}
for the image of rational pairs $(\frac pq, \frac rq)$ under
$\phi_{\omega}$. In \cite[Appendix A]{cas} we have introduced a matrix
representation of the maps $\phi_0$ and $\phi_1$ by
\begin{equation} \label{matrici-01}
    M_0 \coloneqq M_{\phi_0} =
    \begin{pmatrix}
    1 & 0 & 1 \\
    1 & 0 & 0 \\
    0 & 1 & 0
    \end{pmatrix}
    \qquad\text{and}\qquad M_1 \coloneqq M_{\phi_1} =
   \begin{pmatrix}
    1 & 0 & 1 \\
    0 & 1 & 0 \\
    0 & 0 & 1
    \end{pmatrix} ,
\end{equation}
from which we obtain a matrix representation with non-negative
integers coefficients for any combination $\phi_{\omega^j}$ of
$\phi_0$ and $\phi_1$, which we denote by
\[
    M_{\omega^j}\coloneqq
    \begin{pmatrix}
       \rho & \sigma & \tau \\
       \rho_1 & \sigma_1 & \tau_1 \\
       \rho_2 & \sigma_2 & \tau_2
    \end{pmatrix}
\]
not including the dependence on $j$ in the notation for the
coefficients of the matrix if not necessary. Using the matrix
representation we write for all $(x,y)\in \R^2$
\begin{equation}
    \label{matr-repres}
    \phi_{_{\omega^j}}(x,y) = \left( \frac{\rho_1 + \sigma_1 x +
        \tau_1 y}{\rho + \sigma x + \tau y} , \frac{\rho_2 + \sigma_2
        x + \tau_2 y}{\rho + \sigma x + \tau y}\right) ,
\end{equation}
so that in particular it holds
\begin{equation}\label{den-matrix}
    s_{\omega^j}(p,r,q) = \rho q + \sigma p + \tau r
\end{equation}
in \eqref{denom-multi}.

In this paper we begin to study the problem of the speed of the
approximations by considering the first simple classes of real pairs
of numbers with at least one irrational component: the pairs with
finite triangle sequence and the pairs with periodic triangle sequence
$[d,d,d,\,\dots]$ for $d\ge 3$.

\subsection{Real pairs with finite triangle sequence}
\label{sec:finite-approx}

Let $(\alpha,\beta)\in \trianglecl$, with $\alpha$ or $\beta$
irrational, have finite triangle sequence
$[\alpha_0,\,\ldots,\,\alpha_k]$, then as proved in Lemma
\ref{lemma:key}, in particular see \eqref{eq:backSigma}, there exists
$\omega \in \set{0,1}^*$ and $\xi \in [0,1]$ such that
$(\alpha,\beta) = \phi_{\omega}(\xi,0)$ (recall that
$(\xi,0)=\phi_2(\xi,\xi)$). As remarked in Section \ref{sec:approx},
we can definitively consider the problem \eqref{prob-speed} for
rational pairs $(\frac{m_j}{s_j},\frac{n_j}{s_j})$ with the same
triangle sequence of $(\alpha,\beta)$. It follows that we can consider
rational pairs as obtained in \eqref{denom-multi} with
$\phi_\omega$. Hence we can write
\[
    \Big|\alpha - \frac{m_{\omega}(p,r,q)}{s_{\omega}(p,r,q)}\Big| =
    \Big| \left(\phi_{\omega}(\xi,0)\right)_1 -
    \left(\phi_{\omega}\triple{p}{r}{q}\right)_1\Big|
\]
\[
    \Big| \beta - \frac{n_{\omega}(p,r,q)}{s_{\omega}(p,r,q)}\Big| =
    \Big| \left(\phi_{\omega}(\xi,0)\right)_2 -
    \left(\phi_{\omega}\triple{p}{r}{q}\right)_2\Big|
\]
where the subscripts refer to the first and second component
respectively, and consider $(\frac pq,\frac rq)$ as approximations of
$(\xi,0)$.

Using the matrix representation \eqref{matr-repres} for
$\phi_{\omega}$ we can write
\begin{align*}
    \alpha - \frac{m_{\omega}(p,r,q)}{s_{\omega}(p,r,q)} &=
    \frac{\rho_1 + \sigma_1 \xi}{\rho + \sigma \xi} - \frac{\rho_1 q +
      \sigma_1 p + \tau_1 r}{\rho q + \sigma p + \tau r} =\\
    &= \frac{(\sigma_1 \tau- \sigma \tau_1)\xi r +(\sigma_1 \rho -
      \sigma \rho_1)(\xi q-p)+(\tau \rho_1 - \tau_1 \rho) r}{(\rho +
      \sigma \xi) (\rho q + \sigma p + \tau r)}
\end{align*}
and analogously
\[
    \beta - \frac{n_{\omega}(p,r,q)}{s_{\omega}(p,r,q)} =
    \frac{(\sigma_2 \tau- \sigma \tau_2) \xi r +(\sigma_2 \rho -
      \sigma \rho_2)(\xi q-p)+(\tau \rho_2 - \tau_2 \rho) r}{(\rho +
      \sigma \xi) (\rho q + \sigma p + \tau r)}
\]
Using that for all $\xi \in \R$ there exist two sequences $(p_j)_j$ of
integers and $(q_j)_j$ of positive integers with $(p_j,q_j)=1$, such
that $|\xi q_j - p_j| \le \frac 1{q_j}$ for all $j$, and letting
$r_j= 0$ for all $j$, so that it holds
$\triple{p_j}{r_j}{q_j} \in \trianglecl \cap \Q^2$, we obtain that
there exist two sequences $(p_j)_j$ of integers and $(q_j)_j$ of
positive integers such that
\[
    \Big| \alpha - \frac{m_{\omega}(p_j,0,q_j)}{s_{\omega}(p_j,0,q_j)}
    \Big| \le c \frac{1}{q_j s_{\omega}(p_j,0,q_j)}
    \quad\text{and}\quad
    \Big| \beta -
    \frac{n_{\omega}(p_j,0,q_j)}{s_{\omega}(p_j,0,q_j)}\Big| \le c
    \frac{1}{q_j s_{\omega}(p_j,0,q_j)}
\]
where the constant $c$ does not depend on $p_j$ and $q_j$. Therefore
choosing the two sequences $(p_j)_j$ and $(q_j)_j$ as before, we have
that for all $\eps>0$
\begin{align*}
    \liminf_{s\to \infty} &s^{(4-\eps)} \Big(\inf_{m\in \Z} \Big|
    \alpha -\frac ms\Big|\Big) \Big(\inf_{n\in \Z} \Big| \beta -\frac
    ns\Big|\Big) \le \liminf_{j \to \infty} s_j^{(4-\eps)}
    \left|\alpha-\frac{m_j}{s_j}\right|
    \left|\beta-\frac{n_j}{s_j}\right| \le\\
    &\le \lim_{j \to \infty} s_{\omega}(p_j,0,q_j)^{(4-\eps)}
    \left|\alpha-\frac{m_{\omega}(p_j,0,q_j)}{s_{\omega}(p_j,0,q_j)}\right|
    \left|\beta-\frac{n_{\omega}(p_j,0,q_j)}{s_{\omega}(p_j,0,q_j)}\right|
    \le\\
    &\le c^2 \lim_{j \to \infty}
    \frac{s_{\omega}(p_j,0,q_j)^{(2-\eps)}}{q_j^2} = 0 ,
\end{align*}
since by \eqref{den-matrix} holds $s_{\omega}(p_j,0,q_j) \le c' q_j$
for a suitable constant $c'$. This argument also implies that in this
case $(\alpha,\beta)$ is is not a Bad~$(i_1,i_2)$ pair, for all
admissible $(i_1,i_2)$ (see \cite{schmidt-book} and
\cite{dioph-review}).

\subsection{Real pairs with periodic triangle sequence
  $[d,d,d,\,\dots]$ for $d\ge 3$}
\label{sec:periodic-approx}

Let $(\alpha,\beta)\in \trianglecl$, with $\alpha$ or $\beta$
irrational, have triangle sequence $[d,d,d,\,\dots]$ for $d\ge 3$. These
pairs correspond to fixed points for the Triangle Map $T$ defined in
Section \ref{sec:trianglemaps} (all fixed points of $T$ are obtained
by considering also the real pairs with triangle sequence
$[d,d,d,\,\dots]$ for $d=0,1,2$).

It is shown in \cite{garr} that if $(\alpha,\beta)$ has triangle
sequence $[d,d,d,\,\dots]$ then $\beta = \alpha^2$ and
$\alpha\in (0,1)$ is the largest root of the polynomial
$P(t)=t^3 + dt^2 +t-1$. It follows that $\alpha$ is a cubic number,
and $\alpha$ and $\beta$ are in the same cubic number field. We remark
that the polynomial $P(t)$ has three real roots for $d\ge 3$, and only
one real root for $d=0,1,2$. Hence here we only consider the simplest
case for the roots of $P(t)$.

In Section \ref{sec:approx} we have constructed approximations of
$(\alpha,\beta)$ by rational pairs in the Triangular tree with
representation $([d,d,\,\dots,d],[2])$, where we recall that instead
of $[2]$ one could choose any fixed continued fraction
expansion $[a_1,\dots,a_n]$. Then using \eqref{matrici-01} we consider
the matrix
\[
    M_d \coloneqq M_1^d M_0 = \begin{pmatrix} 1 & d & 1 \\ 1 & 0 & 0\\
        0 & 1 & 0
\end{pmatrix}
\]
which as in \eqref{matr-repres} represents the map
\[
    \phi_{\omega} (x,y) = \left( \frac{1}{1 + d x + y} , \frac{x}{1 +
        d x + y}\right) \quad \text{with}\quad \omega=(11\cdots1 0) \in
    \{ 0,1\}^{d+1}
\]
and use approximations of $(\alpha,\beta)$ of the form
\begin{equation} \label{periodic-approx} \triple{m_k}{n_k}{s_k} =
    \phi_{_{M_d^k}}\left(\frac 12, 0\right) \eqqcolon
    \phi_{\overline{\omega}^k}\left(\frac 12, 0\right)
\end{equation}
where $\overline{\omega}^k \in \{0,1\}^{k(d+1)}$ is the string
obtained by concatenating $k$ copies of $\omega$. In matrix
representation, \eqref{periodic-approx} can be written as
\[
    \begin{pmatrix} s_k\\ m_k\\ n_k \end{pmatrix} =
    M_d^k \begin{pmatrix} 2\\ 1\\ 0 \end{pmatrix}
\]
and we can then use linear algebra to study the properties of the approximations.

\begin{remark}\label{rem-diff-point}
    As shown in \eqref{converg-ok} all points in $\trianglecl$
    converge to $(\alpha,\beta)$ under repeated applications of
    $\phi_\omega$. Therefore different approximations of
    $(\alpha,\beta)$ can be constructed as in \eqref{periodic-approx}
    by using a sequence $(\frac{p_k}{q_k},\frac{r_k}{q_k})$ of
    possibly different rational pairs instead of the fixed pair
    $(\frac 12,0)$.
\end{remark}

The matrix $M_d$ has characteristic polynomial
$p_d(\lambda)=\lambda^3 -\lambda^2 -d \lambda -1$ and, for $d\ge 3$,
distinct eigenvalues $\lambda_1$, $\lambda_2$, and $\lambda_3$
satisfying
\begin{equation}
    \label{eigen-md}
    \lambda_1 =\frac{1}{\alpha_2} < \lambda_2=\frac{1}{\alpha_1} < 0 <
    1< \lambda_3 = \frac 1 \alpha \quad \text{with}\quad \lambda_3 >
    |\lambda_1| \ge 1 > |\lambda_2|
\end{equation}
where $\alpha_1<\alpha_2<0<\alpha<1$ are the roots of $P(t)$, and
eigenvectors
\[
    v_1 = \begin{pmatrix}1 \\[0.1cm] \alpha_2\\[0.1cm]
        \alpha_2^2 \end{pmatrix} , \quad v_2 = \begin{pmatrix}1
        \\[0.1cm] \alpha_1\\[0.1cm] \alpha_1^2 \end{pmatrix} , \quad
    v_3 = \begin{pmatrix}1 \\[0.1cm] \alpha \\[0.1cm]
        \beta \end{pmatrix}.
\]
We also recall that for $d\ge 3$ fixed, and given $\alpha$ the largest
root of $P(t)$, it holds
\begin{equation}\label{roots-relaz}
    \alpha_1 = \frac{-(d+\alpha)-\sqrt{(d+\alpha)^2-\frac 4\alpha}}{2}
    \quad\text{and}\quad \alpha_2 =
    \frac{-(d+\alpha)+\sqrt{(d+\alpha)^2-\frac 4\alpha}}{2}.
\end{equation}
For simplicity of notation in the following we let
\[
    h(d,\alpha) \coloneqq \sqrt{(d+\alpha)^2-\frac 4\alpha}
\]
We are now ready to prove the following result.

\begin{proposition}\label{prop:cubic-uno}
    Let $(\alpha,\beta)\in \trianglecl$, with $\alpha$ or $\beta$
    irrational, have triangle sequence $[d,d,d,\,\dots]$ for $d\ge
    3$. There exist a constant $c(\alpha,d)$, functions
    $f^i_{\alpha,d}:\Z^3 \to \R$ with $i=1,2$ given by
    \[
        \begin{aligned}
            f^1_{\alpha,d}(q,p,r) &=  \frac{h(d,\alpha) \left(d+3\alpha+h(d,\alpha)\right) (1+\alpha^2(d+2\alpha))}{4\alpha} \Big[ q \left(3-\alpha+\alpha^2 h(d,\alpha)\right) +{} \\[0.2cm]
            &\hspace{1cm}+ p \left(-2-2d\alpha-2\alpha h(d,\alpha)\right) +  r \left(-3\alpha-d+h(d,\alpha)\right) \Big]\\[0.2cm]
            f^2_{\alpha,d}(q,p,r) &=   -\frac{h(d,\alpha) \left(d+3\alpha-h(d,\alpha)\right) (1+\alpha^2(d+2\alpha))}{4\alpha} \Big[ q \left(3-\alpha-\alpha^2 h(d,\alpha)\right) +{} \\[0.2cm]
            &\hspace{1cm} + p \left(-2-2d\alpha+2\alpha
              h(d,\alpha)\right) + r
            \left(-3\alpha-d-h(d,\alpha)\right) \Big]
        \end{aligned}
    \]
    and functions $g^i_{\alpha,d}:\Z^3 \to \R$ with $i=1,2,3$ given by
    \[
        g^1_{\alpha,d}(q,p,r) = \frac{d+3\alpha+h(d,\alpha)}{2} \left[
          q \frac{(-\alpha)\left(d+\alpha+h(d,\alpha)\right)}{2} + p
          \frac{d-\alpha+h(d,\alpha)}{2} + r\right]
    \]
    \[
        g^2_{\alpha,d}(q,p,r) = \frac{-d-3\alpha+h(d,\alpha)}{2}
        \left[ q \frac{(-\alpha)\left(d+\alpha-h(d,\alpha)\right)}{2}
          + p \frac{d-\alpha-h(d,\alpha)}{2} + r\right]
    \]
    \[
        g^3_{\alpha,d}(q,p,r) = -h(d,\alpha) \left[q \frac 1 \alpha +p
          (d+\alpha) + r\right]
    \]
    such that for all $k\ge 1$ and all
    $(\frac pq,\frac rq)\in \trianglecl$ the rational pair
    $(\frac ms,\frac ns)= \phi_{M_d^k}(\frac pq,\frac rq)$ satisfies
    \begin{equation}\label{prima-stima}
        \begin{aligned}
            & \alpha - \frac ms = c(\alpha,d) \frac{f^1_{\alpha,d}(q,p,r) \lambda_1^k  + f^2_{\alpha,d}(q,p,r) \lambda_2^k}{g^1_{\alpha,d}(q,p,r) \lambda_1^k +g^2_{\alpha,d}(q,p,r) \lambda_2^k+g^3_{\alpha,d}(q,p,r) \lambda_3^k}\\[0.3cm]
            & \beta - \frac ns=  \frac{c(\alpha,d)}{2} \frac{(\alpha-d+h(d,\alpha)) f^1_{\alpha,d}(q,p,r) \lambda_1^k  + (\alpha-d-h(d,\alpha))\ f^2_{\alpha,d}(q,p,r) \lambda_2^k}{g^1_{\alpha,d}(q,p,r) \lambda_1^k +g^2_{\alpha,d}(q,p,r) \lambda_2^k+g^3_{\alpha,d}(q,p,r) \lambda_3^k}\\[0.3cm]
        \end{aligned}
    \end{equation}
    where $\lambda_1$, $\lambda_2$, and $\lambda_3$ are defined in
    \eqref{eigen-md}.
\end{proposition}

\begin{proof}
    If $(\alpha,\beta)$ has triangle sequence $[d,d,d,\,\dots]$ then
    it is a fixed point of the Triangle Map $T$ with
    $\beta = \alpha^2$ and, in particular using \eqref{eq:T-jumpS},
    $(\alpha,\beta) = S^{d+1}(\alpha,\beta) = S|_{\Gamma_0}\circ
    S|_{\Gamma_1}^d (\alpha,\beta)$ for the map $S$ defined in
    \eqref{eq-slow}. It follows that
    $(\alpha,\beta) = \phi_{M_d^k}(\alpha,\beta)$ for all $k\ge 1$. We
    can thus write for fixed $k\ge 1$ and rational pair
    $(\frac pq,\frac rq)\in \trianglecl$
\begin{equation} \label{intanto1}
\begin{aligned}
    \alpha - \frac ms &=  \left(\phi_{M_d^k}(\alpha,\beta)\right)_1 - \left(\phi_{M_d^k}\triple{p}{r}{q}\right)_1 = \\
    &= \frac{(\rho \sigma_1 -\rho_1 \sigma) (\alpha q-p) + (\tau
      \sigma_1 - \tau_1 \sigma)(\alpha r -\beta p)+ (\rho
      \tau_1-\rho_1 \tau)(\beta q-r)}{(\rho+\sigma \alpha + \tau
      \beta) (\rho q+ \sigma p + \tau r)}
\end{aligned}
\end{equation}
using the matrix representation \eqref{matr-repres}. Let $\Lambda$ be
the diagonal matrix
$\Lambda= {\textrm{diag}}(\lambda_1,\lambda_2,\lambda_3)$, then
$M_d^k = \Pi \Lambda^k \Pi^{-1}$ where $\Pi$ is the matrix with
columns the eigenvectors $v_1,v_2$ and $v_3$. We use this
representation to have an explicit form for the terms in
\eqref{intanto1}, where the dependence on $k$ appears in the power of
the eigenvalues and is of course hidden in the coefficients of
$M_d^k$. In what follows, all not explained steps are just
straightforward computations. We have
\[
    \Pi = \begin{pmatrix} 1 & 1 & 1 \\[0.1cm] \alpha_2 & \alpha_1 &
        \alpha \\[0.1cm] \alpha_2^2 & \alpha_1^2 & \beta
    \end{pmatrix} \quad\text{and}\quad \Pi^{-1} =
    \frac{1}{\det(\Pi)} \begin{pmatrix} \alpha_1 \beta - \alpha_1^2
        \alpha & \alpha_1^2 - \beta & \alpha-\alpha_1 \\[0.1cm]
        \alpha_2^2 \alpha - \alpha_2 \beta & \beta - \alpha_2^2 &
        \alpha_2-\alpha\\[0.1cm] \alpha_1^2 \alpha_2 - \alpha_1
        \alpha_2^2 & \alpha_2^2 - \alpha_1^2 & \alpha_1-\alpha_2
\end{pmatrix}
\]
with
$\det(\Pi) = \alpha_1 \alpha (\alpha-\alpha_1) + \alpha_2 \alpha
(\alpha_2-\alpha) + \alpha_1 \alpha_2 (\alpha_1-\alpha_2)$.\\[0.15cm]
We first consider the denominator of \eqref{intanto1}. Let us write
the first row of $M_d^k$ as a function of $\alpha$ and $d$. It holds
\begin{equation}\label{prima-riga}
    \begin{aligned}
        \rho &=  \frac{1}{\det(\Pi)} \Big[ \lambda_1^k \alpha \alpha_1 (\alpha-\alpha_1) +  \lambda_2^k \alpha \alpha_2 (\alpha_2-\alpha) +  \lambda_3^k \alpha_1 \alpha_2 (\alpha_1-\alpha_2)\Big]\\[0.2cm]
        \sigma &=  \frac{1}{\det(\Pi)} \Big[ \lambda_1^k (\alpha_1^2 - \alpha^2) + \lambda_2^k (\alpha^2 - \alpha_2^2) + \lambda_3^k (\alpha_2^2 - \alpha_1^2)\Big]\\[0.2cm]
        \tau &= \frac{1}{\det(\Pi)} \Big[ \lambda_1^k (\alpha-\alpha_1) +
        \lambda_2^k (\alpha_2-\alpha) + \lambda_3^k
        (\alpha_1-\alpha_2)\Big]
    \end{aligned}
\end{equation}
then we obtain
\begin{equation}\label{int-den1}
    \rho+\sigma \alpha + \tau \beta = \lambda_3^k
\end{equation}
and
\begin{align*}
    rho q+ \sigma p + \tau r = \frac{1}{\det(\Pi)} &\left[ \lambda_1^k\Big(q\alpha_1 \alpha(\alpha-\alpha_1) + p(\alpha_1^2-\alpha^2) + r(\alpha-\alpha_1)\Big) \right. +\\
    & \hspace{0.5cm} +\left. \lambda_2^k \Big( q\alpha_2
      \alpha(\alpha_2-\alpha) +
      p(\alpha^2-\alpha_2^2) + r(\alpha_2-\alpha)\Big) +  \right.\\
    & \hspace{1cm} +\left.\lambda_3^k \Big( q\alpha_1
      \alpha_2(\alpha_1-\alpha_2) + p(\alpha_2^2-\alpha_1^2) +
      r(\alpha_1-\alpha_2)\Big) \right]
\end{align*}
which using \eqref{roots-relaz} gives
\begin{equation}\label{int-den2}
    \rho q+ \sigma p + \tau r =\frac{1}{\det(\Pi)} \Big(
    g^1_{\alpha,d}(q,p,r) \lambda_1^k +g^2_{\alpha,d}(q,p,r)
    \lambda_2^k+g^3_{\alpha,d}(q,p,r) \lambda_3^k \Big)
\end{equation}
with the functions $ g^i_{\alpha,d}(q,p,r)$ defined in the
statement. For the numerator of \eqref{intanto1} let us write the
second row of $M_d^k$ as a function of $\alpha$ and $d$. It holds
\[
\begin{aligned}
    \rho_1 &= \frac{1}{\det(\Pi)} \Big[ \lambda_1^k (\alpha-\alpha_1) +  \lambda_2^k (\alpha_2-\alpha) +  \lambda_3^k (\alpha_1-\alpha_2)\Big]\\
    \sigma_1 &= \frac{1}{\det(\Pi)} \Big[ \lambda_1^k \alpha_2 (\alpha_1^2 - \alpha^2) + \lambda_2^k \alpha_1 (\alpha^2 - \alpha_2^2) + \lambda_3^k \alpha (\alpha_2^2 - \alpha_1^2)\Big]\\
    \tau_1 &= \frac{1}{\det(\Pi)} \Big[ \lambda_1^k \alpha_2
    (\alpha-\alpha_1) + \lambda_2^k \alpha_1(\alpha_2-\alpha) +
    \lambda_3^k \alpha (\alpha_1-\alpha_2)\Big]
\end{aligned}
\]
and then one can write the numerator of \eqref{intanto1} as
\[
    \frac{1}{(\det(\Pi))^2} \Big(f^1_{\alpha,d}(q,p,r) \lambda_1^k
    \lambda_3^k + f^2_{\alpha,d}(q,p,r) \lambda_2^k
    \lambda_3^k+f^3_{\alpha,d}(q,p,r) \lambda_1^k \lambda_2^k\Big)
\]
with $f^3_{\alpha,d}(q,p,r)\equiv 0$, and $f^1_{\alpha,d}(q,p,r)$ and
$f^2_{\alpha,d}(q,p,r)$ defined as in the statement. This gives the
first of \eqref{prima-stima}.

Now we can repeat the argument for the second term in \eqref{prima-stima} to get for fixed $k\ge 1$ and rational pair $(\frac pq,\frac rq)\in \trianglecl$
\begin{equation} \label{intanto2}
\begin{aligned}
    \beta - \frac ns &= \left(\phi_{M_d^k}(\alpha,\beta)\right)_2 - \left(\phi_{M_d^k}\triple{p}{r}{q}\right)_2 = \\
    &= \frac{(\rho \sigma_2 -\rho_2 \sigma) (\alpha q-p) + (\tau
      \sigma_2 - \tau_2 \sigma)(\alpha r -\beta p)+ (\rho
      \tau_2-\rho_2 \tau)(\beta q-r)}{(\rho+\sigma \alpha + \tau
      \beta) (\rho q+ \sigma p + \tau r)}
\end{aligned}
\end{equation}
Equations \eqref{prima-riga}, \eqref{int-den1} and \eqref{int-den2}
give the denominator of \eqref{intanto2}. For the numerator we write
the third row of $M_d^k$ as a function of $\alpha$ and $d$, finding
\[
    \frac{1}{(\det(\Pi))^2} \Big(\frac{\alpha-d+h(d,\alpha)}{2}
    f^1_{\alpha,d}(q,p,r) \lambda_1^k \lambda_3^k +
    \frac{\alpha-d-h(d,\alpha)}{2} f^2_{\alpha,d}(q,p,r) \lambda_2^k
    \lambda_3^k\Big)
\]
From this we obtain the second equation in \eqref{prima-stima}.
\end{proof}

\begin{corollary}
    \label{cor-approx-periodic}
    Let $(\alpha,\beta)\in \trianglecl$, with $\alpha$ or $\beta$
    irrational, have triangle sequence $[d,d,d,\,\dots]$ for $d\ge 3$,
    and consider the approximations
    $(\frac{m_k}{s_k},\frac{n_k}{s_k})$ defined in
    \eqref{periodic-approx}. Then
    \[
        \lim_{k\to \infty} s_k^\eta
        \left|\alpha-\frac{m_k}{s_k}\right|
        \left|\beta-\frac{n_k}{s_k}\right| = 0
    \]
    for all
    $\eta<2 \big(1- \frac{\log |\lambda_1|}{\log \lambda_3}\big)$ if
    $d\ge 4$, where we are using \eqref{eigen-md}, and for for all
    $\eta<4$ if $d=3$.
\end{corollary}

\begin{proof}
    We apply Proposition \ref{prop:cubic-uno} with $(q,p,r) = (2,1,0)$
    together with the relations \eqref{eigen-md}. It follows
    \[
        s_k^\eta \left|\alpha-\frac{m_k}{s_k}\right|
        \left|\beta-\frac{n_k}{s_k}\right| \le c'(\alpha,d) \Big(
        g^3_{\alpha,d}(2,1,0) \lambda_3^k + o(\lambda_3^k)
        \Big)^{\eta-2} \Big( f^1_{\alpha,d}(2,1,0) \lambda_1^k +
        o(\lambda_1^k) \Big)^2
    \]
    where $g^3_{\alpha,d}(2,1,0)$ and $f^1_{\alpha,d}(2,1,0)$ do not
    vanish for $d\ge 4$. The result for this case immediately follows.\\[0.15cm]
    The case $d=3$ is particular since $\alpha=-1+\sqrt{2}$ is a
    quadratic irrational. It also follows $\beta = 3-2\sqrt{2}$,
    $h(3,\alpha) = \sqrt{2}$, and
    \[
        \lambda_1 = -1 , \quad \lambda_2 = 1-\sqrt{2} , \quad
        \lambda_3 = 1+\sqrt{2} .
    \]
    Moreover
    \[
        f^1_{\alpha,3}(2,1,0)= c''(\alpha,3) \Big( 4-8\alpha
        +2(\alpha^2-\alpha) h(3,\alpha) \Big) = 0
    \]
    and $f^2_{\alpha,3}(2,1,0)$ and $g^3_{\alpha,3}(2,1,0)$ do not
    vanish. Since $\lambda_2 = - \lambda_3^{-1}$, it follows that
    \[
        s_k^\eta \left|\alpha-\frac{m_k}{s_k}\right|
        \left|\beta-\frac{n_k}{s_k}\right| \le c'''(\alpha,3) \Big(
        g^3_{\alpha,3}(2,1,0) \lambda_3^k + o(\lambda_3^k)
        \Big)^{\eta-2} \lambda_3^{-2k}
    \]
    and the thesis follows.
\end{proof}

The result for $d\ge 4$ can be improved if we change the construction
of the approximations as explained in Remark~\ref{rem-diff-point}. For
the moment we leave this problem and the study of the speed of the
approximations for other real pairs to future research.


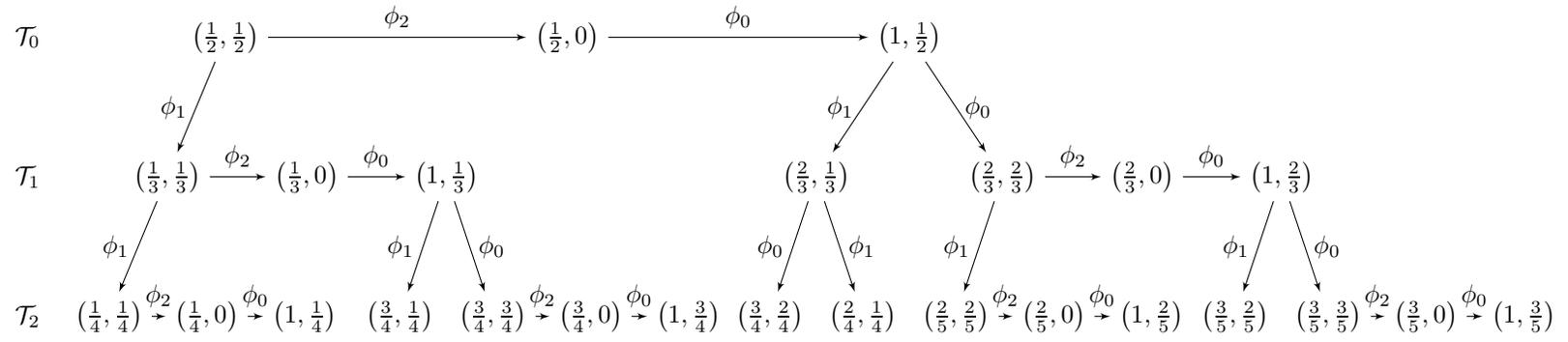
\begin{sidewaysfigure}
    \vspace{15cm}
    \begin{tikzpicture}[
    level 1/.style = {sibling distance=2.5cm},
    level 2/.style = {sibling distance=2.5cm},
    level 3/.style = {sibling distance=4cm},
    level 4/.style = {sibling distance=2cm},
    level distance = 3cm,
    edge from parent/.style = {draw,-latex'},
    every node/.style       = {-latex'},
    scale=0.64
    ]
    \def\d{4.35}
    \def\hd{0} 
    \def\hdd{-0.9}   
    \def\ld{2.5}   

    \node at (-4.25,0) {$\mathcal{T}_0$};
    \node at (-4.25,-3) {$\mathcal{T}_1$};
    \node at (-4.25,-6) {$\mathcal{T}_2$};

    \node {$\triple{1}{1}{2}$} [grow=down]
    child {
      node {$\triple{1}{1}{3}$}
      child {
        node {$\triple{1}{1}{4}$}
        child [grow=right] {
          node at (\hdd,0) {$\singy{1}{4}{0}$}
          child [grow=right] {
            node at (\hdd,0) {$\singx{1}{1}{4}$} [grow=down]
            edge from parent node[above] {$\phi_0$}}
          edge from parent node[above] {$\phi_2$}}
        edge from parent node[left] {$\phi_1$}}
      child [grow=right] {
        node at (\hd,0) {$\singy{1}{3}{0}$}
        child [grow=right] {
          node  at (\hd,0) {$\singx{1}{1}{3}$} [grow=down]
          child {
            node {$\triple{3}{1}{4}$}
            edge from parent node[left] {$\phi_1$}}
          child {
            node {$\triple{3}{3}{4}$}
            child [grow=right] {
              node at (\hdd,0) {$\singy{3}{4}{0}$}
              child [grow=right] {
                node at (\hdd,0) {$\singx{1}{3}{4}$} [grow=down]
                edge from parent node[above] {$\phi_0$}}
              edge from parent node[above] {$\phi_2$}}
            edge from parent node[right] {$\phi_0$}}
          edge from parent node[above] {$\phi_0$}
        }
        edge from parent node[above] {$\phi_2$}
      }
      edge from parent node[anchor=east] {$\phi_1$}
    }
    child [grow=right] {
      node at (\d,0) {$\singy{1}{2}{0}$}
      child [grow=right] {
        node  at (\d,0) {$\singx{1}{1}{2}$} [grow=down]
        child {
          node {$\triple{2}{1}{3}$} [grow=down]
          child {
            node {$\triple{3}{2}{4}$}
            edge from parent node[left] {$\phi_0$}}
          child {
            node {$\triple{2}{1}{4}$}
            edge from parent node[right] {$\phi_1$}}
          edge from parent node[left] {$\phi_1$}}
        child {
          node {$\triple{2}{2}{3}$}
          child  {
            node {$\triple{2}{2}{5}$}
            child [grow=right] {
              node at (\hdd,0) {$\singy{2}{5}{0}$}
              child [grow=right] {
                node at (\hdd,0) {$\singx{1}{2}{5}$} [grow=down]
                edge from parent node[above] {$\phi_0$}}
              edge from parent node[above] {$\phi_2$}}
            edge from parent node[left] {$\phi_1$}}
          child [grow=right] {
            node at (\hd,0) {$\singy{2}{3}{0}$}
            child [grow=right] {
              node at (\hd,0) {$\singx{1}{2}{3}$} [grow=down]
              child {
                node {$\triple{3}{2}{5}$}
                edge from parent node[left] {$\phi_1$}}
              child {
                node {$\triple{3}{3}{5}$}
                child [grow=right] {
                  node at (\hdd,0) {$\singy{3}{5}{0}$}
                  child [grow=right] {
                    node at (\hdd,0) {$\singx{1}{3}{5}$}
                    edge from parent node[above] {$\phi_0$}}
                  edge from parent node[above] {$\phi_2$}}
                edge from parent node[right] {$\phi_0$}}
              edge from parent node[above] {$\phi_0$}
            }
            edge from parent node[above] {$\phi_2$}
          }
          edge from parent node[right] {$\phi_0$}}
        edge from parent node[above] {$\phi_0$}
      }
      edge from parent node[above] {$\phi_2$}
    };
\end{tikzpicture}
    \caption{The levels $\mathcal{T}_{0}$, $\mathcal{T}_{1}$ and $\mathcal{T}_{2}$ of the Triangular tree generated through the rules (R1)-(R4).}\label{fig:tree-maps}
\end{sidewaysfigure}

\end{document}